\documentclass[12pt,reqno]{amsart}
\usepackage{amsmath, amsthm, amssymb}

\usepackage{comment}
\usepackage{color}
\usepackage{amscd}

\advance \topmargin by -\headheight
\advance \topmargin by -\headsep
     
\evensidemargin \oddsidemargin
\newtheorem{theorem}{Theorem}[section]
\newtheorem{lemma}[theorem]{Lemma}
\newtheorem{corollary}[theorem]{Corollary}

\newtheorem{proposition}[theorem]{Proposition}

\theoremstyle{definition}

\theoremstyle{remark}
\newtheorem{remark}[theorem]{Remark}

\numberwithin{equation}{section}

\newcommand{\mmod}[1]{\,\,\text{mod}\,\,#1}

\def\bfa{{\mathbf a}}
\def\bfb{{\mathbf b}}

 \def\bfe{{\mathbf e}}

\def\bft{{\mathbf t}}
\def\bfu{{\mathbf u}}
\def\bfv{{\mathbf v}}

\def\bfx{{\mathbf x}}
\def\bfy{{\mathbf y}}
\def\bfz{{\mathbf z}}

\def\atil{{\tilde{a}}}

\def\qtil{{\tilde{q}}}

\def\bfatil{{\widetilde{\bfa}}}
\def\bfvtil{{\widetilde{\bfv}}}
\def\bfxtil{{\widetilde{\bfx}}}

\def\bfxhat{{\widehat{\bfx}}}

\def\calA{{\mathcal A}}

\def\calB{{\mathcal B}}

\def\calD{{\mathcal D}}
\def\calE{{\mathcal E}}

\def\calM{{\mathcal M}}

\def\calO{{\mathcal O}}

\def\A{{\mathbb A}}
\def\C{{\mathbb C}}\def\N{{\mathbb N}}\def\P{{\mathbb P}}
\def\R{{\mathbb R}}
\def\Z{{\mathbb Z}}\def\Q{{\mathbb Q}}

\def\grm{{\mathfrak m}}\def\grM{{\mathfrak M}}

\def\grS{{\mathfrak S}}

\def\alp{{\alpha}} \def\bfalp{{\boldsymbol \alpha}}
\def\bet{{\beta}}  \def\bfbet{{\boldsymbol \beta}}

\def\Gam{{\Gamma}}
\def\phitil{{\widetilde \phi}}
\def\del{{\delta}} \def\deltil{{\tilde \delta}} \def\Del{{\Delta}}
\def\zet{{\zeta}}  

\def\tet{{\theta}}  
\def\vartet{{\vartheta}}
\def\kap{{\kappa}}
\def\lam{{\lambda}} \def\Lam{{\Lambda}}

\def\sig{{\sigma}}

\def\ome{{\omega}} 
\def\d{{\partial}}
\def\eps{\varepsilon}

\def\d{{\,{\rm d}}}

\def\meas{{\rm meas}}
\def\rank{{\rm rank}}

\def\sing{{\rm sing}}

\def\Pic{{\rm Pic}}

\def\grMyd{{\grM^{'\bfy}}}
\def\grMy{{\grM^{\bfy}}}

\def\Xbar{{\overline{X}}}
\def\Br{{\rm Br}}
\def\eff{{\rm eff}}
\def\vol{{\rm vol}}
\def\Gal{{\rm Gal}}
\def\Val{{\rm Val}}
\def\Frob{{\rm Frob}}
\def\Ntil{{\widetilde{N}}}
\def\kbar{{\overline{k}}}
\def\Xbar{{\overline{X}}}
\def\et{{\rm \acute{e}t}}
\def\Pey{{\rm Peyre}}

\newenvironment{blue}{\color{blue}}{}

\specialcomment{comblue}{\begin{blue}}{\end{blue}}
\excludecomment{comblue}
\excludecomment{com}

\begin{document}
\title[Manin's conjecture for certain biprojective hypersurfaces]{Manin's conjecture for certain biprojective hypersurfaces}
\author[Damaris Schindler]{Damaris Schindler}
\address{Hausdorff Center for Mathematics, Endenicher Allee 62, 53115 Bonn, Germany}
\email{damaris.schindler@hcm.uni-bonn.de}


\date{\today}

\begin{abstract}
Using the circle method, we count integer points on complete intersections in biprojective space in boxes of different side length, provided the number of variables is large enough depending on the degree of the defining equations and certain loci related to the singular locus. Having established these asymptotics we deduce asymptotic formulas for rational points on such varieties with respect to the anticanonical height function. In particular, we establish a conjecture of Manin for certain smooth hypersurfaces in biprojective space of sufficiently large dimension.
\end{abstract}

\maketitle

\section{Introduction}

The goal of this paper is to study the distribution of rational points on complete intersections in biprojective space. In particular, we prove a conjecture of Manin for certain smooth hypersurfaces in biprojective space of sufficiently large dimension depending mostly on the degree of the defining equation.\par
To state our main result we introduce some notation. Let $n_1$ and $n_2$ be
positive integers and write $\bfx = (x_1,\ldots, x_{n_1})$ and $\bfy=
(y_1,\ldots, y_{n_2})$. Let $F_1(\bfx;\bfy),\ldots, F_R(\bfx;\bfy)$ be $R$
bihomogeneous polynomials with integer coefficients, all of bidegree
$(d_1,d_2)$. They define a variety $X$ in biprojective space $\P_\Q^{n_1-1}\times
\P_\Q^{n_2-1}$ given by
\begin{equation}\label{eqn0}
F_i(\bfx;\bfy)=0,\quad 1\leq i\leq R.
\end{equation}
Assuming $n_i >Rd_i$ for $i=1,2$, we introduce the following height function on rational points of
$\P_\Q^{n_1-1}\times \P_\Q^{n_2-1}$. For a point $(\bfx;\bfy)$ with integer
coordinates such that $\gcd (x_1,\ldots, x_{n_1})=1$ and $\gcd (y_1,\ldots,
y_{n_2})=1$ we define
\begin{equation*}
H(\bfx;\bfy)= \left( \max_{1\leq i\leq n_1} |x_i|^{n_1-Rd_1}\right) \left( \max_{1\leq j\leq n_2}
  |y_j|^{n_2-Rd_2}\right).
\end{equation*}
We wish to understand the number of rational points of bounded height on $X$
with respect to this height function. It may happen that this counting
function is dominated by points lying on a proper closed subvariety of
$X$. Hence, we will construct a Zariski-open subset $U\subset X$ and count
points lying in $U$ only. More precisely, let $N_{U,H}(P)$ be the number of
points $(\bfx;\bfy)\in U(\Q)$ with $H(\bfx;\bfy)\leq P$.\par
Before we state our main theorem, we need to introduce certain singular
loci. Let $V_1^*\subset \A_\C^{n_1+n_2}$ be the variety given by
\begin{equation}\label{eqnrank1}
\rank \left( \frac{\partial F_i(\bfx;\bfy)}{\partial
    x_j}\right)_{\substack{1\leq i\leq R\\ 1\leq j\leq n_1}}<R.
\end{equation}
Analogously, we define $V_2^*$ to be the affine variety given by
\begin{equation}\label{eqnrank2}
\rank \left( \frac{\partial F_i(\bfx;\bfy)}{\partial
    y_j}\right)_{\substack{1\leq i\leq R\\ 1\leq j\leq n_2}}<R.
\end{equation}

\begin{theorem}\label{thm3}
Assume that $d_1,d_2\geq 2$. Let $F_i(\bfx;\bfy)$ be a system of bihomogeneous
polynomials as above with
\begin{equation}\label{eqnthm3}
n_1+n_2-\max\{ \dim V_1^*, \dim V_2^*\} > 3\cdot 2^{d_1+d_2} d_1d_2 R^3.
\end{equation}
Then there is a Zariski-open subset $U\subset X$ such that
\begin{equation*}
N_{U,H} (P)= (4\zet (n_1-Rd_1)\zet (n_2-Rd_2))^{-1}\sig P \log P + C_1 P+ O(P^{1-\eta}),
\end{equation*}
for some real number $C_1$ and some $\eta >0$. The constant $\sig$ is the
leading constant predicted by the circle method for the number of integer solutions to the system of equations
(\ref{eqn0}), where the real density is to be taken with respect to the box $[-1,1]^{n_1+n_2}$.
\end{theorem}

We remark that restricting our counting function to an open subset $U$ is
necessary in this theorem. For example consider the hypersurface given by 
\begin{equation*}
F(\bfx;\bfy)= x_1^{d_1}y_1^{d_2}+\ldots +x_n^{d_1}y_n^{d_2}=0,
\end{equation*}
with $d_1,d_2\geq 2$. In this case $V_1^*$ and $V_2^*$ are both given by 
\begin{equation*}
x_iy_i=0,\quad 1\leq i\leq n,
\end{equation*}
such that we have $\dim V_1^*=\dim V_2^*=n$. Hence our Theorem \ref{thm3}
implies the existence of an open subset $U$ with $N_{U,H}(P)\sim c P\log P$,
for some constant $c$, as soon as $n$ is sufficiently large depending on $d_1,d_2$. Consider the rational points of height bounded by $P$
in this hypersurface with $x_1=0$ and $y_2=\ldots =y_n=0$. Their contribution
is of order $P^{\frac{n-1}{n-d_1}}$, which is larger than the main term in Theorem \ref{thm3}.\par
The open subset $U$ in Theorem \ref{thm3} is explicitly described in section 4. It is a product of two open subsets $U_1\times U_2$ with $U_i$ an open subset of affine $n_i$-space for $i=1,2$. More precisely, some point $\bfx\in \A_\C^{n_1}$ is contained in $U_1$ if the variety in affine $n_2$-space given by $F_i(\bfx;\bfy)=0$ for $1\leq i\leq R$ with $\bfx$ considered as fixed, is sufficiently non-singular in the sense of Birch's work \cite{Bir1961}.\par
It is interesting to interpret our main result in the case $R=1$ of hypersurfaces. In \cite{FMT89} Manin conjectured that for Fano manifolds $X$ with Zariski-dense rational points $X(\Q)$
(excluding some cases) an asymptotic behaviour of the form
\begin{equation}\label{eqn0.1}
N_{U,H}(P)\sim c P(\log P)^{\rank (\Pic X)-1},
\end{equation}
should hold, where $H$ is an anticanonical height function. Furthermore, Peyre \cite{Pey95} has given an interpretation and prediction for the leading constant $c$, which we call from now on $c_\Pey$.\par
So far, there are only very few cases of subvarieties of biprojective
space known that show the predicted asymptotic behaviour. For the case of a single
hypersurface of bidegree $(d_1,d_2)=(1,1)$ there is work of Robbiani
\cite{Rob01} proving the desired asymptotic for the variety given by
$x_0y_0+\ldots +x_sy_s=0$, as soon as $s\geq 3$. Using a classical form of the
circle method, Spencer \cite{Spe09} has simplified the proof and extended the
result to $s\geq 2$. There is an independent proof given by Browning
\cite{Bro11} in the case $s=2$, which uses asymptotics for certain
correlations of the divisor function. Furthermore, Le Boudec succeeds in
\cite{Boud13} to provide sharp upper and lower bounds for the counting function
$N_{U,H}(P)$ associated to the threefold in biprojective space given by $x_0y_0^2+x_1y_1^2+x_2y_2^2=0$.\par
We compare Theorem \ref{thm3} with the conjectured formula
(\ref{eqn0.1}) in the case $R=1$. Assume that we are given a smooth hypersurface $X\subset
\P_\Q^{n_1-1}\times \P_\Q^{n_2-1}$ satisfying the conditions of Theorem
\ref{thm3}. In the next section (see Lemma \ref{geolem1b}) we show that the
condition (\ref{eqnthm3}) is automatically satisfied if $X$ is smooth and both
$n_1$ and $n_2$ are sufficiently large. Exercise II.8.3 b) of \cite{Hart} shows that the canonical bundle
on $\P_\Q^{n_1-1}\times \P_\Q^{n_2-1}$ is given by $\calO (-n_1,-n_2)$. By the
adjunction formula (see Prop II.8.20 in \cite{Hart}) we obtain
\begin{equation*}
-\ome_X \cong \calO_X (n_1-d_1,n_2-d_2).
\end{equation*}
Our assumptions in Theorem \ref{thm3} certainly imply that $n_1-d_1\geq 1$ and that
$n_2-d_2\geq 1$. Note that then the set of global sections of
$\calO_X(n_1-d_1,n_2-d_2)$ is generated by monomials of bidegree
$(n_1-d_1,n_2-d_2)$. Such a choice of a set of generators defines an embedding
into projective space, which shows that $-\ome_X$ is very ample.
Hence $X$ is
indeed a Fano variety, and our height function $H$ introduced at the beginning
of this section is an anticanonical height function.\par
In the next section we determine the Picard group of a smooth complete intersection in biprojective space of dimension at least three, see Theorem \ref{geothm1}. In particular we obtain $\Pic X \cong \Z^2$, and hence we have $\rank (\Pic X) = 2$. This shows that
our Theorem \ref{thm3} is compatible with Manin's conjecture for smooth
hypersurfaces in biprojective space. In section 3 we show that the leading constant in Theorem \ref{thm3} is compatible with Peyre's prediction in \cite{Pey95}. This leads to the following theorem.

\begin{theorem}\label{thmManin}
Assume that $d_1,d_2\geq 2$. Let $X$ be a smooth hypersurface in biprojective space $\P_\Q^{n_1-1}\times \P_\Q^{n_2-1}$ of bidegree $(d_1,d_2)$ such that
\begin{equation*}
\min \{ n_1,n_2\} > 1+ 3\cdot 2^{d_1+d_2}d_1d_2.
\end{equation*}
Then Manin's conjecture holds for some Zariski-open subset $U$ of $X$ and the leading constant $c=c_\Pey$ in the asymptotic formula (\ref{eqn0.1}) is the one predicted by Peyre \cite{Pey95}.
\end{theorem}

In the calculation of Peyre's constant $c_\Pey$ one has to compute a Tamagawa measure of the set of adelic points of $X$ cut out by the Brauer group $\Br X $ of $X$. In the appendices of Colliot-Th\'el\`ene and Katz in \cite{PooVol03} it is shown that the Brauer group of a smooth complete intersection in projective space of dimension at least 3 is trivial. The proof also applies to the biprojective setting and implies that the Brauer group of $X$ is trivial as soon as $X$ is a smooth complete intersection in biprojective space with $\dim X\geq 3$, see Proposition \ref{Brauer} in section 2.

\medskip

Our proof of Theorem \ref{thm3} relies on previous work of the
author \cite{bihomforms}. It again makes use of the circle method in
combination with the hyperbola method with weights, which was recently
developed by Blomer and Br\"udern \cite{BloBru13}.\par

The structure of this paper is as follows. After providing some geometric
preliminaries in the next section, we show in section 3 that our leading
constant in Theorem \ref{thm3} is the one predicted by Peyre in
\cite{Pey95} and deduce Theorem \ref{thmManin}. In the fourth section we state our supplementary theorems on counting functions associated to the system of equations (\ref{eqn0}), which we prove in the following sections using the circle method. In particular, in section 5 we apply
Weyl-differencing fibre-wise to the system of polynomials (\ref{eqn0}) and
deduce a form of Weyl-inequality for the corresponding exponential sum. Section 6 and section 7
contain most of the circle method analysis. In section 8 we deduce from
this the main theorems of section 4. The following section 9 is used to
apply the techniques developed by Blomer and Br\"udern to our counting problem
and deduce Theorem \ref{thm3} using the previously mentioned circle method theorems.\par

For some real valued functions $f(P_1,P_2)$ and $g(P_1,P_2)$ we write in the following
$f(P_1,P_2) = O(g(P_1,P_2))$ if there exist positive constants $C$ and
$C_0$ such that $|f(P_1,P_2)|\leq C g(P_1,P_2)$ for all $P_1\geq C_0$ and
$P_2\geq C_0$.\par
We write $\Val (\Q)$ for the set of valuations of $\Q$, and $\Q_\nu$ for the completion of $\Q$ at a place $\nu\in \Val(\Q)$. Furthermore $|\cdot |_\nu$ is the standard $\nu$-adic metric on $\Q_\nu$. We write $\d x_{\nu}$ for the Haar measure on $\Q_\nu$ which is the standard Lebesgue measure for the infinite place and for a finite place $p$ normalized in a way such that $\int_{\Z_p}\d x_p=1$.

\textbf{Acknowledgements.} The author would like to thank Prof. T. D. Wooley
for suggesting this area of research, the referee for his or her comments and Prof. T. D. Browning for useful discussions. The author is grateful to Prof. Salberger for providing the proof of Theorem \ref{geothm1} and for useful comments.

\section{Geometric Preliminaries}
First we state a well-known lemma on the intersection of a closed subvariety
with an ample divisor, which we need in the following several times.

\begin{lemma}\label{geolem1}
Let $W$ be a smooth variety, and $Z\subset W$ be a closed irreducible
subvariety, and $D$ an effective divisor on $W$. Then every irreducible
component of $D\cap Z$ has dimension at least $\dim Z -1$. Furthermore, if $D$
is ample, $W$ complete over some algebraically closed field, and the dimension
of $Z$ is at least one, then the intersection $D\cap Z$ is non-empty. 
\end{lemma}

\begin{proof}
The first statement is for example a consequence of equation (*) in
\cite{Shaf}, p. 238, where we choose $x$ a closed point in the intersection
of $D\cap Z$ if this is not empty. By the Nakai-Mo\v ishezon criterion for
ampleness (see p. 262 in \cite{Shaf}) one has 
\begin{equation*}
(D^r.Z)>0,
\end{equation*}
if $D$ is an ample divisor on a complete variety $W$ and $Z$ an irreducible
subvariety of dimension $\dim Z=r$. This implies in particular that $D\cap
Z\neq \emptyset$ if the dimension of $Z$ is positive.
\end{proof}

In the following we set $W= \P_\C^{n_1-1}\times
\P_\C^{n_2-1}$. We note that for a smooth hypersurface $X\subset W$ the loci $V_1^*$ and $V_2^*$ as defined in the introduction
cannot be too large. 

\begin{lemma}\label{geolem1b}
Assume that $d_1,d_2\geq 2$ and that $X\subset W$ is given by a single bihomogeneous equation $F(\bfx;\bfy)=0$ of bidegree $(d_1,d_2)$. Assume that $X$ is smooth. Then we have
\begin{equation*}
\dim V_i^* \leq \max\{n_1,n_2,n_1+n_2-n_i+1\}
\end{equation*}
for $i=1,2$.
\end{lemma}

\begin{proof}
Let $V_i$ be the variety in biprojective space given by (\ref{eqnrank1}) for $i=1$ and given by (\ref{eqnrank2}) for $i=2$. Then we certainly have
\begin{equation*}
\dim V_i^* \leq \max \{ n_1,n_2, \dim V_i+2\},
\end{equation*}
for $i=1,2$. Hence it is sufficient to bound $\dim V_1 \leq n_2-1$ and $\dim V_2\leq n_1-1$.\par
Let $H_j$ be the subvariety in $W$ given by  $\frac{\partial F}{\partial y_j}=0$ for $1\leq j\leq n_2$. Then the singular locus $X_\sing$ of $X$ in biprojective space is given by
\begin{equation*}
X_\sing = V_1\cap \left(\cap_{j=1}^{n_2}H_j\right).
\end{equation*}
Assume that $\dim V_1 \geq n_2$. We note that each $H_j$ is either equal to the whole biprojective space or an ample divisor since we have assumed $d_1,d_2\geq 2$. Hence Lemma \ref{geolem1} implies that $\dim V_1 \cap H_1 \geq n_2-1$. After intersecting with all the other $H_j$ we obtain
\begin{equation*}
\dim \left(V_1\cap \left(\cap_{j=1}^{n_2}H_j\right) \right) \geq n_2-n_2 = 0,
\end{equation*}
and the intersection is non-empty by Lemma \ref{geolem1}. This is a contradiction to $X$ being smooth, and hence $\dim V_1 \leq n_2-1$. Since the same argument holds for $V_2$, this proves the lemma.
\end{proof}

We keep the notation $W= \P_\C^{n_1-1}\times \P_\C^{n_2-1}$ and fix effective ample
divisors $D_1,\ldots, D_k$. For some $1\leq i\leq k$ write $X_i= \cap_{j=1}^i
D_j$ and $X= X_k$. Set $X_0=W$ and assume that $X=\cap_{j=1}^k D_j$ is a smooth complete intersection of codimension $k$ in $W$. Then all the intermediate intersections $X_i$ are also complete intersections and of codimension $i$. This is for example a consequence of Lemma \ref{geolem1}. Note that the $X_i$ need not be smooth, but they are all Cohen-Macaulay, see for example Proposition II.8.23 in \cite{Hart}.\par

\begin{lemma}\label{geolem2}
Let $0\leq i\leq k$ and $D$ be an ample divisor on $X_i$. Assume that $\dim X_k\geq 3$. Then
\begin{equation}\label{geoeqn1}
H^1(X_i,\calO(-D))= H^2(X_i,\calO (-D))=0,
\end{equation}
for all $0\leq i\leq k$. 

\end{lemma}

\begin{proof}
We use descending induction starting with $i=k$. Note that $X_k$ is smooth by assumption, and hence Kodaira's vanishing theorem applies and gives the desired result since $\dim X_k\geq 3$ (see e.g. Remark III.7.15 in \cite{Hart}).\par
Next assume that $i<k$ and that we already have established the vanishing (\ref{geoeqn1}) for ample divisors on $X_{i+1}$. We consider on $X_i$ the exact sequence of $\calO_{X_i}$-modules
\begin{equation}\label{geoeqn2}
0\rightarrow \calO_{X_i}(-D_{i+1})\rightarrow \calO_{X_i}\rightarrow \calO_{X_{i+1}}\rightarrow 0.
\end{equation}
After twisting with $\calO_{X_i}(-D-(r-1)D_{i+1})$ for some $r\geq 1$ and taking the associated long cohomology sequence, we obtain the exact sequence
\begin{equation}\label{geoeqn3}
\begin{split}
H^1(X_i,\calO(-D-rD_{i+1}))&\rightarrow H^1(X_i, \calO(-D-(r-1)D_{i+1}))\\
\rightarrow H^1(X_{i+1}, \calO(-D-(r-1)D_{i+1}))& \rightarrow
H^2(X_i,\calO(-D-rD_{i+1})) \\ \rightarrow H^2(X_i, \calO(-D-(r-1)D_{i+1}))&\rightarrow H^2(X_{i+1}, \calO(-D-(r-1)D_{i+1})).
\end{split}
\end{equation}
By induction hypothesis and since $D+(r-1)D_{i+1}$ is ample for $r\geq 1$, we have
\begin{equation*}
H^j(X_{i+1}, \calO(-D-(r-1)D_{i+1}))=0,\quad j=1,2.
\end{equation*}
Next we apply Serre duality to the cohomology groups on $X_i$. Recall that all the $X_i$ are Cohen-Macaulay and equidimensional. Write $l_i=\dim X_i$ and let $\ome_{X_i}^0$ be the dualizing sheaf of $X_i$. 
Hence Corollary III.7.7 in \cite{Hart} implies that
\begin{equation*}
H^1(X_i,\calO (-D-rD_{i+1})) \cong H^{l_i-1}(X_i, \calO(D+rD_{i+1})\otimes \ome_{X_i}^0)',
\end{equation*}
where $'$ denotes the dual vector space.\par
Next we apply Serre's vanishing theorem (see Theorem III.5.2 in \cite{Hart}). This implies that there is some $r_0=r_0(X_i)$ such that for all $r\geq r_0$ one has
\begin{equation*}
H^{l_i-1}(X_i,\calO (D+rD_{i+1})\otimes \ome_{X_i}^0)=0.
\end{equation*}
Since we have assumed $\dim X_k\geq 3$ the same holds for the cohomology groups $H^{l_i-2}$. Hence, by Serre duality we have
\begin{equation*}
H^1(X_i,\calO (-D-rD_{i+1}))=H^2(X_i,\calO(-D-rD_{i+1}))=0,
\end{equation*}
for $r\geq r_0$. Now the exact sequence (\ref{geoeqn3}) implies that 
\begin{equation*}
H^j(X_i,\calO (-D-(r-1)D_{i+1}))=0,\quad j=1,2,
\end{equation*}
for $r\geq r_0$. Now induction on $r$ shows that 
\begin{equation*}
H^1(X_i,\calO(-D))= H^2(X_i,\calO(-D))=0,
\end{equation*}
as desired.
\end{proof}

With the help of Lemma \ref{geolem2} we can now determine the Picard group of $X$.

\begin{theorem}\label{geothm1}
Let $X$ be as above a smooth complete intersection in $W$ of dimension at least $3$. Then the restriction homomorphism 
\begin{equation*}
\Pic W\rightarrow \Pic X
\end{equation*}
is an isomorphism, and $\Pic X \cong \Z\times \Z$.
\end{theorem}

\begin{proof}
First we note that by Example A.9.28  (p. 560)
of \cite{BomGub} one has
\begin{equation*}
\Pic (\P_K^{n_1-1}\times \P_K^{n_2-1}) \cong \Z^2,
\end{equation*}
for any field $K$.\par
Next Lemma \ref{geolem2} implies that
\begin{equation*}
H^1(X_i,\calO(-D_{i+1}))=H^2(X_i,\calO(-D_{i+1}))=0,
\end{equation*}
for $0\leq i\leq k$. Since $X_i$ is Cohen-Macaulay and of dimension at least three, it is of depth $\geq 3$ in all its closed points. 
Hence we can apply \cite{SGA2}, Exp. XII, Cor 3.6 to the variety $X_i$ and the divisor $D_{i+1}$. Therefore, the homomorphism $\Pic X_i\rightarrow \Pic X_{i+1}$ is an isomorphism for $0\leq i\leq k-1$. Composing all these isomorphisms
\begin{equation*}
\Pic W \rightarrow \Pic X_1\rightarrow \ldots \rightarrow \Pic X_k
\end{equation*}
gives the result of this theorem.

\end{proof}

Next we note that Lemma \ref{geolem2} also implies that all the intermediate intersections $X_i$ are connected.

\begin{lemma}\label{geolem3}
The variety $X_i$ is connected for all $0\leq i\leq k$.
\end{lemma}

\begin{proof}
We proof this by induction on $i$. Note that $X_0=W$ is connected since $H^0(\P_\C^{n_1-1}\times \P_\C^{n_2-1},\calO_W)=\C$.\par
The exact sequence of sheaves (\ref{geoeqn2}) implies that the sequence
\begin{equation*}
H^0(X_i,\calO_{X_i})\rightarrow H^0(X_{i+1},\calO_{X_{i+1}})\rightarrow H^1(X_i,\calO_{X_i}(-D_{i+1}))
\end{equation*}
is exact. Since the divisor $D_{i+1}$ is ample, Lemma \ref{geolem2} implies that 
\begin{equation*}
H^1(X_i,\calO_{X_i}(-D_{i+1}))=0.
\end{equation*}
Therefore the first map in the above sequence is surjective
\begin{equation*}
H^0(X_i,\calO_{X_i})\twoheadrightarrow H^0(X_{i+1},\calO_{X_{i+1}}),
\end{equation*}
and $H^0(X_{i+1},\calO_{X_{i+1}})=\C$. 

\end{proof}

The appendices at the end of \cite{PooVol03} (see Corollary A.2)
show that the Brauer-Manin obstruction for a smooth complete intersection in
$\P_k^n$ with $\dim X\geq 3$ and $k$ a number field, is vacuous. The proof
contained in this work also applies to complete intersections in biprojective
space, and gives the following result.

\begin{proposition}[Analogue of Proposition A.1 in \cite{PooVol03}]\label{Brauer}
Let $k$ be a number field and $X$ be a smooth complete
intersection in $\P_k^{n_1-1}\times \P_k^{n_2-1}$ of effective ample divisors satisfying $\dim X\geq 3$.
Then the natural map $\Br k\rightarrow \Br X$ is an
isomorphism.
\end{proposition}

\begin{proof}
First let $k$ be an algebraically closed field of characteristic zero,
and set $V= \P_k^{n_1-1}\times
\P_k^{n_2-1}$. Let $Y$ be given by $F_i(\bfx;\bfy)=0$, $1\leq i\leq R$ for a
system of bihomogeneous polynomials of bidegree $(d_1^{(i)},d_2^{(i)})$. Let $H_i$ be
given by $F_i(\bfx;\bfy)=0$. Then we claim that $V\setminus H_i$ is
affine. For this consider the map
\begin{equation*}
\phi: \P_k^{n_1-1}\times \P_k^{n_2-1}\hookrightarrow \P_k^{N_1}\times
\P_k^{N_2}\hookrightarrow \P_k^N,
\end{equation*}
where the first map is the product of a Veronese embedding
$\P_k^{n_1-1}\hookrightarrow \P_k^{N_1}$ of degree $d_1^{(i)}$ and a Veronese
embedding of the second factor of degree $d_2^{(i)}$, followed by a Segre
embedding. Then $\phi(H_i)$ is given by one linear equation. Hence $\phi
(V)\setminus \phi(H_i)$ is affine as desired.\par

Let $l$ be a prime invertible in $k$ and let $H_\et^i$ denote \'{e}tale cohomology. Then Corollary B.5 in \cite{PooVol03} implies that the restriction map
\begin{equation}\label{Br1}
H_\et^i (\P_k^{n_1-1}\times \P_k^{n_2-1}, \Z/l\Z) \rightarrow H_\et^i (Y,\Z/l\Z)
\end{equation}
is an isomorphism for $i<n_1+n_2-2-R$ and injective for $i=n_1+n_2-2-R$.\par

Note that in our situation of a smooth complete intersection in biprojective
space, $\Br Y$ is torsion. To show that $\Br Y$ is trivial it is hence enough
to prove that that the $l$-torsion part $(\Br Y)[l]=0$ for all primes $l$.\par
We assume for a moment that $n_i\geq 2$ for $i=1,2$. Otherwise Proposition
\ref{Brauer} reduces to Proposition A.1 in \cite{PooVol03}. As in Appendix A in \cite{PooVol03} one can consider the commutative diagram

\begin{equation*}\begin{array}{ccccccccc}
0&\to& \Pic (V)/l &\to& H_\et^2 (V,\Z/l\Z) && &&\\
&&\downarrow&&\downarrow&&&&\\
0&\to& \Pic (Y)/l&\to& H_\et^2 (Y,\Z/l\Z)&\to&
(\Br Y)[l]&\to&0\\
\end{array}\label{due}
\end{equation*}
whose rows are exact.
For $\dim Y \geq 3$, the right vertical map is an
isomorphism by equation (\ref{Br1}). Furthermore, the top horizontal map is an
isomorphism since both groups are of rank two over $\Z/l\Z$.
This implies $(\Br Y)[l]=0$ for all primes $l$ as desired.\par
\medskip
To adapt the proof of Proposition A.1 in \cite{PooVol03} to the biprojective
setting, we have to check the following ingredients. Let $X$ be as in Proposition
\ref{Brauer}, denote by $\kbar$ an algebraic closure of $k$, let $G= \Gal (\kbar/k)$ and $\Xbar = X
\times_k \kbar$. Then we need to check that $X$ is geometrically connected,
that $\Pic X \rightarrow (\Pic \Xbar)^G$ is an isomorphism, that $H^1(k,\Pic
\Xbar)=0$ and that $\Br \Xbar=0$. The last of these follows directly from the
above comments.\par
Lemma \ref{geolem2} implies that $X$ is
geometrically connected since $\dim X \geq 3$. 
By Theorem \ref{geothm1} there is an ismorphism $\Pic \Xbar \cong \Z\times \Z$, and hence $H^1(k,\Pic \Xbar)$ is
trivial. Furthermore, Theorem \ref{geothm1} implies that
the restriction map 
\begin{equation*}
\Pic \left(\P_\kbar^{n_1-1}\times \P_\kbar^{n_2-1}\right)\rightarrow \Pic \Xbar
\end{equation*}
is an isomorphism, and hence $\Pic X \rightarrow (\Pic \Xbar)^G$ is an
isomorphism as explained in \cite{PooVol03}, Appendix A.  

\end{proof}

\section{Interpretation of the leading constant}

In this section we consider a single bihomogeneous polynomial $F(\bfx;\bfy)=0$ of bidegree $(d_1,d_2)$ which defines a hypersurface $X\subset \P_\Q^{n_1-1}\times \P_\Q^{n_2-1}$. Suppose that the assumptions of Theorem \ref{thm3} are satisfied. In particular, we have $n_i-d_i\geq 2$ for $i=1,2$ and hence the anticanonical sheaf 
\begin{equation*}
\ome_X^{-1} \cong \calO_X (n_1-d_1,n_2-d_2)
\end{equation*}
is very ample. We let $s_1,\ldots, s_q$ be the global sections of $\calO_X(n_1-d_1,n_2-d_2)$ given by all monomials in $(\bfx;\bfy)$ of bidegree
$(d_1,d_2)$. They generate the ring of global sections $\Gam (X, \calO(n_1-d_1,n_2-d_2))$, and define an
adelic metric on $\calO_X (n_1-d_1,n_2-d_2)$ and hence a height function on $X(\Q)$ given by
\begin{equation*}
H(\bfx;\bfy)= \prod_{\nu \in \Val(\Q)} \max_{i,j}|x_i^{n_1-d_1}y_j^{n_2-d_2}|_\nu.
\end{equation*}
If $\bfx$ and $\bfy$ are both given by reduced integer vectors, then this is the same as saying
\begin{equation*}
H(\bfx;\bfy)= \left(\max_i|x_i|^{n_1-d_1}\right) \left( \max_j |y_j|^{n_2-d_2}\right),
\end{equation*}
which is nothing else than the anticanonical height function introduced in the last section. According to Peyre the leading constant in equation (\ref{eqn0.1}) should be of the form
\begin{equation}\label{Pey1}
c_\Pey = \alp (X) \bet (X) \lim_{s\rightarrow 1}( (s-1)^{\rank (\Pic X)} L(s,\chi_{\Pic (\Xbar)})) \tau_H (X(\A_\Q)^{\Br}).
\end{equation}
This expression can for example be found in Chapter VI, section 5 of
\cite{Jahnel}. In the rest of this section we define each factor separately,
and compute them for $X$ as above. We follow mainly the formulation and analysis of the constant in \cite{Jahnel}, in \cite{Pey95} and \cite{PeyTsch00}.\par
Recall that we have an isomorphism $\Pic X \cong \Pic (\P_\Q^{n_1-1}\times \P_\Q^{n_2-1} )\cong \Z^2$. 
The hyperplanes $H_1: x_1=0$ and $H_2: y_1=0$ generate $\Pic (\P_\Q^{n_1-1}\times \P_\Q^{n_2-1} )$ freely, and hence also $\Pic X$. Using additive notation for the divisor class group, we know that 
\begin{equation*}
-K_X = (n_1-d_1)H_1+(n_2-d_2)H_2,
\end{equation*}
with $K_X$ the class of the canonical divisor. We use the classes $H_1$ and $H_2$ to identify $\Pic X$ with the lattice $\Z^2$ in $\R^2$. The real cone of effective divisors of $X$ is then given by
\begin{equation*}
\Lam_\eff (X) = \{t_1H_1+t_2H_2: t_1,t_2\geq 0\} \subset \R^2.
\end{equation*}
Let $\Lam_\eff^\vee (X) \subset (\R^2)^\vee$ be the dual of the effective
cone.
Then the constant $\alp (X)$ is defined to be 
\begin{align*}
\alp (X)&= \rank (\Pic X) \vol \{ z\in \Lam_\eff^\vee | \langle z,-K_X\rangle \leq 1\}\\
&= 2 \vol \{ t_1,t_2\in \R| t_1,t_2\geq 0 \mbox{ and } (n_1-d_1)t_1 + (n_2-d_2)t_2 \leq 1\} \\ &= \frac{1}{(n_1-d_1)(n_2-d_2)}.
\end{align*}
Next we come to the constant $\bet (X)$. As usual, write $\Xbar = X\times \bar{\Q}$. Then the constant $\bet (X)$ is defined to be the cardinality of the first Galois cohomology group
\begin{equation*}
\bet (X)= \sharp H^1 (\Gal (\bar{\Q}/\Q), \Pic \Xbar).
\end{equation*}
In our case $\Pic \Xbar \cong \Z^2$ with trivial Galois action, hence $\bet (X)=1$.\par
We turn to the third term in the product in equation (\ref{Pey1}). Since the
absolute Galois group acts trivially on $\Pic (\Xbar)$, one has
$L(s,\chi_{\Pic (\Xbar)}) = \zet (s)^2$, and hence
\begin{equation*}
\lim_{s\rightarrow 1} (s-1)^{\rank (\Pic X)} L(s,\chi_{\Pic (\Xbar)}) =1.
\end{equation*}

Proposition \ref{Brauer} shows that the Brauer group is trivial in our
setting. Hence we have $X(\A_\Q)^\Br = X(\A_\Q)$. Furthermore our variety $X$ is projective, and therefore we have
$X(\A_\Q)= \prod_{\nu \in \Val(\Q)} X(\Q_\nu)$. In this
situation the Tamagawa measure $\tau_H(X(\A_\Q))$ factors as 
\begin{equation*}
\tau_H(X(\A_\Q)) = \prod_{\nu \in \Val (\Q)} \tau_\nu (X(\Q_\nu)).
\end{equation*}
In the following we define the local measures $\tau_\nu$. For a finite place $p$ this is given as in Definition 5.20 in \cite{Jahnel} by 
\begin{equation*}
\tau_p= \det (1-p^{-1}\Frob_p|\Pic \Xbar^{I_p})\ome_p,
\end{equation*}
with $\ome_p$ the Tamagawa measure as defined in \cite{Pey95} and where
we write $I_p$ for the inertia group. In our case this simplifies to
\begin{equation*}
\tau_p= (1-p^{-1})^{2} \ome_p.
\end{equation*}
For the infinite place one directly sets $\tau_\infty= \ome_\infty$. Next we give a description of $\ome_\nu$ for any place $\nu \in \Val(\Q)$. Let $U_{1,1}$ be the standard open subset of $\P^{n_1-1}\times \P^{n_2-1}$ given by $x_1y_1\neq 0$ and write $n=n_1+n_2-3$. Let $(\bfx;\bfy) \in X$ be a point with $\partial F/\partial y_{n_2}(\bfx;\bfy)\neq 0$. Consider the morphism
\begin{align*}
\rho: X_{\Q_\nu} \cap U_{1,1} &\rightarrow \A_{\Q_\nu}^n\\
(\bfx;\bfy)&\mapsto \left(\frac{x_2}{x_1},\ldots, \frac{x_{n_1}}{x_1},\frac{y_2}{y_1},\ldots, \frac{y_{n_2-1}}{y_1}\right).
\end{align*}
By the $\nu$-adic implicit function theorem the map $\rho$ induces an analytic isomorphism of some open subset $V\subset X$ in the $\nu$-adic topology with $\rho (V)$. Furthermore, $\rho$ induces a map of coherent sheaves
\begin{equation*}
\ome (\rho): \rho^* \ome_{\A^n_{\Q_\nu}/\Q_\nu} \rightarrow \ome_{X\cap U_{1,1}/\Q_\nu}.
\end{equation*}

\begin{com}
To make this morhpism explicit, we introduce some more notation. 
\end{com}

given by
\begin{equation*}
\ome(\rho) (\d u_2 \wedge \ldots \wedge \d v_{n_2-1}) = \d u_2\wedge \ldots \wedge \d v_{n_2-1}.
\end{equation*}
Here we write $u_2,\ldots, u_{n_1},v_2,\ldots, v_{n_2-1}$ for the local coordinates on $\A_{\Q_\nu}^n$.\par
Next we observe that we have an isomorphism 
\begin{equation*}
\ome_{X\cap U_{1,1}}\rightarrow \calO_X (-n_1+d_1,-n_2+d_2)|_{U_{1,1}}.
\end{equation*}
On the Zariski-open subset given by $\partial F/\partial y_{n_2}\neq 0$ this is locally induced by
\begin{equation*}
\d \left( \frac{x_2}{x_1}\right)\wedge \ldots \wedge \d \left(\frac{y_{n_2-1}}{y_1}\right) \mapsto \frac{\partial F}{\partial y_{n_2}}\left(1,\frac{x_2}{x_1},\ldots, \frac{y_{n_2}}{y_1}\right) x_1^{-n_1+d_1}y_1^{-n_2+d_2}.
\end{equation*}
According to section 2.2.1 of \cite{Pey95} the Tamagawa measure $\ome_\nu$ is given by
\begin{equation*}
\rho_* \ome_\nu = \frac{\d u_{2,\nu}\times \ldots \times \d v_{n_2-1,\nu}}{\max_{1\leq i\leq q}|s_i(\rho^{-1}(u,v))(\ome(\rho)(\d u_2\wedge \ldots \wedge \d v_{n_2-1})|_\nu}.
\end{equation*}

We introduce the local heights
\begin{equation*}
h_\nu^1(\bfx)=\max_{1\leq i\leq n_1}|x_i^{n_1-d_1}|_\nu \quad \mbox{ and }\quad h_\nu^2(\bfy)=\max_{1\leq j\leq n_2}|y_j^{n_2-d_2}|_\nu,
\end{equation*}
and set $h_\nu(\bfx;\bfy)= h_\nu^1(\bfx)h_\nu^2(\bfy)$. We use the vector notation $\bfu= (1,u_2,\ldots, u_{n_1})$ and $\bfv= (1,v_2,\ldots, v_{n_2})$. Then we obtain
\begin{equation*}
\ome_\nu = \frac{\d u_{2,\nu}\times \ldots \times \d v_{n_2-1,\nu}}{h_\nu (\bfu;\bfv)\left|\frac{\partial F}{\partial y_{n_2}}(\bfu,\bfv)\right|_\nu},
\end{equation*}
where $v_{n_2}$ is implicitly given by $u_2,\ldots, v_{n_2-1}$. 

For a finite place $p$, the local measure $\ome_p(X(\Q_p))$ is closely related to the usual circle method
density. As usual, we define this local circle method
density $\sig_p$ by
\begin{equation*}
\sig_p = \lim_{l\rightarrow \infty} p^{-l(n_1+n_2-1)}\sharp \{ (\bfx;\bfy)
\mmod p^l : F(\bfx;\bfy) \equiv 0 \mmod p^l\}.
\end{equation*}
 
Then we have the following lemma, which we prove at the end of this section.

\begin{lemma}\label{peylem1}
With the above notation one has
\begin{equation*}
\ome_p (X(\Q_p)) = \frac{(1-p^{-(n_1-d_1)})(1-p^{-(n_2-d_2)})}{(1-p^{-1})^2}
\sig_p.
\end{equation*}
\end{lemma}

Let $\sig_\infty$ be the singular integral for the system of equations (\ref{eqn0}) and with respect to the box $(-1,1)^{n_1}\times (-1,1)^{n_2}$, as defined for example in section 6 in \cite{Bir1961}. Then $\sig_\infty$ is related to the Tamagawa measure of $X(\R)$ in the following way.

\begin{lemma}\label{peylem2}
One has
\begin{equation*}
\tau_\infty (X(\R))= \frac{(n_1-d_1)(n_2-d_2)}{4} \sig_\infty.
\end{equation*}
\end{lemma}

Before we come to the proof of Lemma \ref{peylem1} and Lemma \ref{peylem2} we deduce Theorem \ref{thmManin} from the above and Theorem \ref{thm3}.

\begin{proof}[Proof of Theorem \ref{thmManin}]
Assume that $X\subset \P_\Q^{n_1-1}\times \P_\Q^{n_2-1}$ is a smooth hypersurface given by a bihomogeneous polynomial $F(\bfx;\bfy)$ of bidegree $(d_1,d_2)$ with $d_1,d_2\geq 2$. Then Lemma \ref{geolem1b} implies that 
\begin{equation*}
n_1+n_2-\max \{ \dim V_1^*, \dim V_1^*\} \geq \min \{n_1,n_2\}-1.
\end{equation*}
Recall that we have assumed in Theorem \ref{thmManin} that
\begin{equation*}
\min \{n_1,n_2\} > 1+ 3\cdot 2^{d_1+d_2}d_1d_2.
\end{equation*}
Hence Theorem \ref{thm3} applies to $X$ and delivers an asymptotic formula of the form
\begin{equation}
\label{ddd}
N_{U,H} (P)= (4\zet (n_1-d_1)\zet (n_2-d_2))^{-1}\sig P \log P +O(P),
\end{equation}
for some Zariski-open subset $U$ of $X$. As pointed out in the introduction, the shape of this asymptotic formula is already compatible with Manin's prediction. It remains to show that the leading constant is the one predicted by Peyre.\par
Using Lemma \ref{peylem1} and Lemma \ref{peylem2} together with the description of Peyre's constant in (\ref{Pey1}) and the remarks following it, we can compute the Peyre constant $c_\Pey$ for the hypersurface $X$ as
\begin{align*}
c_\Pey= &\frac{1}{(n_1-d_1)(n_2-d_2)} \prod_p
(1-p^{-(n_1-d_1)})(1-p^{-(n_2-d_2)})\sig_p \\ &\times \frac{(n_1-d_1)(n_2-d_2)}{4} \sig_\infty
\\ = &\frac{1}{4} \zet (n_1-d_1)^{-1} \zet (n_2-d_2)^{-1} \sig_\infty
\prod_p\sig_p.
\end{align*}
This is exactly our leading constant in (\ref{ddd}) coming from Theorem \ref{thm3}.
\end{proof}

\medskip
\subsection{Proof of Lemma \ref{peylem1} and Lemma \ref{peylem2}}

Let $\pi$ be the natural map
\begin{equation*}
\pi: \A_\Q^{n_1+n_2}\setminus (\A_\Q^{n_1}\times \{0\} \cup \{ 0\} \times \A_\Q^{n_2})
\rightarrow \P_\Q^{n_1-1}\times \P_\Q^{n_2-1},
\end{equation*}
and set $W= \pi^{-1}(X)$. Let $(\bfx;\bfy) \in W$ be a smooth (closed) point with
$\partial F/(\partial y_j)(\bfx;\bfy)$ invertible for some $j$ (if one of the derivaties with
respect to some $x_j$ is non-vanishing, then we just interchange notation). Then the Leray form $\ome_L$ on $W$ is given by
\begin{equation*}
\ome_L (\bfx;\bfy)= (-1)^{(n_2-j)} \left( \frac{\partial F}{\partial y_j}\right)^{-1} \d
x_1\wedge \ldots \wedge \d x_{n_1} \wedge \d y_1\wedge \ldots \wedge \widehat{ \d
  y_j} \wedge \ldots \wedge \d y_{n_2} (\bfx;\bfy).
\end{equation*}
For each place $\nu\in \Val(\Q)$ the Leray form induces a local measure $\ome_{L,\nu}$.\par
For a finite place we can relate the Tamagawa measure to a Leray measure via the following lemma, which is a slight modification of Lemma 5.4.6 in \cite{Pey95} to the biprojective situation.

\begin{lemma}\label{peylem1a}
Let $p$ be a finite place, and write
\begin{equation*}
a(p)= (1-p^{-1})^2 (1-p^{-(n_1-d_1)})^{-1}(1-p^{-(n_2-d_2)})^{-1}.
\end{equation*}
Then we have
\begin{equation*}
\int_{\{(\bfx;\bfy)\in W(\Q_p): h_p^1(\bfx)\leq 1,\ h_p^2(\bfy)\leq 1\}}\ome_{L,p}(\bfx;\bfy) = a(p) \ome_p(X(\Q_p)).
\end{equation*}
\end{lemma}

\begin{proof}
We fix an open subset $V\subset X(\Q_p)$ in the $p$-adic topology such that $\left(\frac{x_2}{x_1},\ldots, \frac{x_{n_1}}{x_1},\frac{y_2}{y_1},\ldots, \frac{y_{n_2-1}}{y_1}\right)$ induce a diffeomorphism $\rho$ with the image 
\begin{equation*}
\rho (V)= U \subset \A_{\Q_p}^{n_1+n_2-3} \subset \P_{\Q_p}^{n_1-1}\times \P_{\Q_p}^{n_2-2}.
\end{equation*}
To prove the lemma it is enough to assume that $U$ is of the form $U_1\times U_2$ with $U_1\subset \A_{\Q_p}^{n_1-1}$ and $U_2 \subset \A_{\Q_p}^{n_2-2}$. Then $(x_1,\ldots, x_{n_1},y_1,\ldots, y_{n_2-1})$ define a diffeomorphism of the biaffine cone of $V$ with the product of the affine cones $CU_1\times C U_2$. We assume this diffeomorphism in the following implicitly.\par
Define the functions
\begin{equation*}
g(\bfx;\bfy)= \left| \frac{\partial F}{\partial y_{n_2}}(\bfx;\bfy)\right|_p, \mbox{ and } h(\bfx;\bfy)= h_p^1(\bfx) h_p^2(\bfy).
\end{equation*}
Then we can write
\begin{equation*}
a(p) \ome_p(V)= \int_{U_1\times U_2} \frac{\d x_{2,p}\ldots \d x_{n_1,p}\d y_{2,p}\ldots \d y_{n_2-1,p}}{g \cdot h (1,x_2,\ldots, x_{n_1},1,y_2,\ldots, y_{n_2})},
\end{equation*}
where $y_{n_2}$ is implicitly given by the other coordinates. For a fixed vector $(x_2,\ldots, x_{n_1})\in U_1$ we consider
\begin{equation*}
J(x_2,\ldots, x_{n_1})= \int_{U_2} \frac{\d y_{2,p}\ldots \d y_{n_2-1,p}}{g(1,x_2,\ldots, x_{n_1},1,y_2,\ldots, y_{n_2-1})h_p^2(\bfy)}.
\end{equation*}
Note that we have $g(\bfx;\lam \bfy)= |\lam|_p^{d_2-1} g(\bfx;\bfy)$, and $h_p^2(\lam \bfy)= |\lam|_p^{n_2-d_2}h_p^2(\bfy)$ for $\lam \in \Q_p$. Hence we can apply Lemma 5.4.5 of \cite{Pey95} and obtain
\begin{equation*}
(1-p^{-1})(1-p^{-(n_2-d_2)})^{-1} J(x_2,\ldots, x_{n_1})= \int_{\{ \bfy \in CU_2: h_p^2(\bfy)\leq 1\}} \frac{1}{g}\d y_{1,p}\ldots \d y_{n_2-1,p}.
\end{equation*}
Hence we obtain
\begin{align*}
\ome_p(V) a(p)= &(1-p^{-1})(1-p^{-(n_1-d_1)})^{-1} \\ &\times \int_{U_1}\int_{\{ \bfy\in CU_2: h_p^2(\bfy)\leq 1\}} \frac{ \d x_{2,p}\ldots \d x_{n_1,p}\d y_{1,p}\ldots \d y_{n_2-1,p}}{g h_p^1(\bfx)}.
\end{align*}
Now we interchange the order of integration and obtain after another application of Lemma 5.4.5 of \cite{Pey95} 
\begin{equation*}
a(p)\ome_p(V)= \int_{\{\bfx \in CU_1: h_p^1(\bfx)\leq 1\}} \int_{\{ \bfy \in C U_2: h_p^2(\bfy) \leq 1\}} \frac{1}{g} \d x_{1,p}\ldots \d x_{n_1,p}\d y_{1,p}\ldots \d y_{n_2-1,p}.
\end{equation*}
The last expression is exactly the integral over the Leray measure $\ome_{L,p}(\bfx;\bfy)$.
\end{proof}

For the proof of Lemma \ref{peylem1} we need two more lemmata, which are slight modifications of Lemma 3.2 and Lemma 3.3 in \cite{PeyTsch00}.

\begin{lemma}\label{peylem1b}
Let
\begin{equation*} 
W^*(r)= \{ (\bfx;\bfy)\in (\Z_p/p^r)^{n_1+n_2}: \bfx\not\equiv 0 (p), \bfy \not\equiv 0 (p) \mbox{ and } F(\bfx;\bfy)\equiv 0 \mmod p^r\},
\end{equation*}
and set $N^*(r)= \sharp W^*(r)$. Then there is some $r_0$ such that for all $r\geq r_0$ one has
\begin{equation*}
\int_{\substack{\{ (\bfx;\bfy)\in \Z_p^{n_1+n_2}: \bfx\not\equiv 0 (p)\\ \bfy
    \not\equiv 0 (p), F(\bfx;\bfy)=0\} }}\ome_{L,p}= \frac{N^*(r)}{p^{r(n_1+n_2-1)}}.
\end{equation*}
\end{lemma}

\begin{proof}
For $(\bfx;\bfy)\in \Z_p^{n_1+n_2}$ we write $[\bfx;\bfy]_r$ for the residue class modulo $p^r$. Following the proof of Lemma 3.2 in \cite{PeyTsch00} we start in writing
\begin{align*}
\int_{\substack{\{ (\bfx;\bfy)\in \Z_p^{n_1+n_2}: \bfx\not \equiv 0 (p)\\ \bfy \not \equiv 0 (p), F(\bfx;\bfy)=0\}}}\ome_{L,p} &= \sum_{\substack{ (\bfx;\bfy) \mmod p^r\\ \bfx \not\equiv 0 (p), \bfy \not\equiv 0 (p)}} \int_{\substack{\{ (\bfu;\bfv)\in \Z_p^{n_1+n_2}, [\bfu;\bfv]_r=(\bfx;\bfy)\\ F(\bfx;\bfy)=0\}}}\ome_{L,p}(\bfu;\bfv)\\ &= \sum_{(\bfx;\bfy)\in W^*(r)}\int_{\substack{\{ (\bfu;\bfv)\in \Z_p^{n_1+n_2}, [\bfu;\bfv]_r=(\bfx;\bfy)\\ F(\bfx;\bfy)=0\}}}\ome_{L,p}(\bfu;\bfv).
\end{align*}
Since $X$ is smooth, there is some $r$ sufficiently large such that for any element $(\bfx;\bfy) \in (\Z_p/p^r)^{n_1+n_2}$ with $\bfx\not\equiv 0 \mmod p$ and $\bfy \not\equiv 0 \mmod p$ and with $F(\bfx;\bfy)\equiv 0 \mmod p^r$, the infimum
\begin{equation*}
c= \inf_{i,j} \left( \nu_p \left(\frac{\partial F}{\partial x_i}\right),\nu_p\left(\frac{\partial F}{\partial y_j}\right)\right)
\end{equation*}
is finite and constant on the class defined by $(\bfx;\bfy)$. Assume that $r>c$ and that 
\begin{equation*}
c= \nu_p\left(\frac{\partial F}{\partial y_{n_2}}(\bfx;\bfy)\right),
\end{equation*}
is the minimum.\par
Let $(\bfu;\bfv)\in \Z_p^{n_1+n_2}$ represent $(\bfx;\bfy)$ and let $(\bfz;\bfz')\in \Z_p^{n_1+n_2}$. Then one has
\begin{align*}
F(\bfu+\bfz;\bfv+\bfz')= F(\bfu;\bfv) + \sum_{i=1}^{n_1}\frac{\partial F}{\partial x_i}(\bfu;\bfv) z_i+ \sum_{j=1}^{n_2}\frac{\partial F}{\partial y_j} (\bfu;\bfv) z_j' + G(\bfu,\bfv,\bfz,\bfz'),
\end{align*}
where $G(\bfu,\bfv,\bfz,\bfz')$ is a polynomial such that each term contains at least two factors of $z_i$ or $z_j'$. Hence, for $(\bfz;\bfz')\in (p^r \Z_p)^{n_1+n_2}$ we have
\begin{equation*}
F(\bfu+\bfz;\bfv+\bfz')\equiv F(\bfu;\bfv) \mmod p^{r+c}.
\end{equation*}
Thus, the image of $F(\bfu;\bfv)$ in $\Z_p/p^{r+c}$ only depends on $(\bfu;\bfv)$ modulo $p^r$. We write $F^*(\bfx;\bfy)$ for this value.\par
If $F^*(\bfx;\bfy)\neq 0$, then the inner integral above corresponding to that value of $(\bfx;\bfy)$ is zero and the set
\begin{equation*}
\{(\bfu;\bfv) \mmod p^{r+c}, [\bfu;\bfv]_r= (\bfx;\bfy): F(\bfu;\bfv)\equiv 0 \mmod p^{r+c}\}
\end{equation*}
is empty.\par
If $F^*(\bfx;\bfy)=0$, then Hensel's Lemma shows that there is an isomorphism of the set
\begin{equation*}
\{(\bfu;\bfv) \in \Z_p^{n_1+n_2}, [\bfu;\bfv]_r= (\bfx;\bfy): F(\bfu;\bfv)=0\}
\end{equation*}
and $(u_1,\ldots, u_{n_1},v_1,\ldots, v_{n_2-1})+(p^r\Z_p)^{n_1+n_2-1}$. Hence we have
\begin{align*}
\int_{\substack{\{ (\bfu;\bfv)\in \Z_p^{n_1+n_2}, [\bfu;\bfv]_r=(\bfx;\bfy)\\ F(\bfx;\bfy)=0\}}}\ome_{L,p}(\bfu;\bfv) &= \int_{(u_1,\ldots, v_{n_2-1})+(p^r\Z_p)^{n_1+n_2-1}}p^c \d u_{1,p}\ldots \d v_{n_2-1,p} \\ &= p^{c-r(n_1+n_2-1)}.
\end{align*}
On the other hand we have
\begin{align*}
&p^{-(r+c)(n_1+n_2-1)}\sharp \{ (\bfu;\bfv) \mmod p^{r+c}, [\bfu;\bfv]_r=(\bfx;\bfy): F(\bfu;\bfv)\equiv 0 \mmod p^{r+c}\} \\ &= p^{-(r+c)(n_1+n_2-1)}p^{(n_1+n_2)c}= p^{c-r(n_1+n_2-1)},
\end{align*}
since $F(\bfu;\bfv)$ modulo $p^{r+c}$ only depends on $(\bfx;\bfy)$.\par
The Lemma now follows via summing over all $(\bfx;\bfy)\in W^*(r)$.

\end{proof}

\begin{lemma}\label{peylem1c}
One has
\begin{equation*}
 \int_{\substack{\{ (\bfx;\bfy)\in \Z_p^{n_1+n_2}: \bfx\not \equiv 0 (p)\\ \bfy \not \equiv 0 (p), F(\bfx;\bfy)=0\}}}\ome_{L,p}= (1-p^{-(n_1-d_1)})(1-p^{-(n_2-d_2)})\int_{\{(\bfx;\bfy)\in \Z_p^{n_1+n_2}: F(\bfx;\bfy)=0\}}\ome_{L,p},
\end{equation*}
and
\begin{equation*}
\lim_{r\rightarrow \infty} \frac{N^*(r)}{p^{r(n_1+n_2-1)}}= (1-p^{-(n_1-d_1)})(1-p^{-(n_2-d_2)}) \sig_p.
\end{equation*}
\end{lemma}

\begin{proof}
The first part of the lemma follows from the observation that
\begin{equation*}
\ome_{L,p}(p\bfx;\bfy)= p^{-n_1+d_1}\ome_{L,p}(\bfx;\bfy) \mbox{ and } \ome_{L,p}(\bfx; p\bfy)= p^{-n_2+d_2}\ome_{L,p}(\bfx;\bfy).
\end{equation*}
For the second part of the lemma we recall that
\begin{equation*}
\sig_p= \lim_{r\rightarrow \infty} \frac{ \sharp \{ (\bfx;\bfy) \mmod p^r : F(\bfx;\bfy) \equiv 0 \mmod p^r\}}{p^{r(n_1+n_2-1)}}.
\end{equation*}
Next we assume that $r\geq id_1+jd_2+1$ and consider the set
\begin{align*}
\Ntil (i,j)= &\sharp \{ \bfx\in (p^i\Z_p/p^r)^{n_1}, \bfx\not\equiv 0
(p^{i+1}), \bfy \in (p^j \Z_p/p^r)^{n_2}, \bfy \not\equiv 0 (p^{j+1}),\\ & F(\bfx;\bfy) \equiv 0 \mmod p^r\}.
\end{align*}
Then we have
\begin{align*}
\Ntil (i,j) = \sharp \{ &\bfx \mmod p^{r-i}, \bfx \not\equiv 0 \mmod p, \bfy \mmod
p^{r-j}, \bfy\not\equiv 0 \mmod p,\\ & F(\bfx;\bfy)\equiv 0 \mmod p^{r-id_1-jd_2}\} \\ &= p^{n_1(id_1+j d_2-i)+n_2 (id_1+jd_2-j)}N^*(r-id_1-jd_2).
\end{align*}
Define
\begin{equation*}
N(r)= \sharp \{ \bfx,\bfy \mmod p^r: F(\bfx;\bfy) \equiv 0 \mmod p^r\}.
\end{equation*}
Let $r_0$ be as in Lemma \ref{peylem1b}, and let $I(r)$ be the set of all integer tuples $(i,j)$ such that $r-r_0 <id_1+jd_2\leq r-r_0 +d_1+d_2$. Then we have
\begin{align*}
N(r)&= \sum_{i\geq 0}\sum_{\substack{j\geq 0\\ r-id_1-jd_2\geq r_0}} \Ntil (i,j) \\ &+ O\left(\sum_{(i,j)\in I(r)} \sharp \{ (\bfx;\bfy) \mmod p^r: \bfx \equiv 0 (p^i),\bfy \equiv 0 (p^j)\}\right).
\end{align*}
Since $n_i >d_i$, the error term can be bounded by
\begin{align*}
&\ll_{r_0} r  \max_{(i,j)\in I(r)}p^{n_1(r-i)+n_2(r-j)} \\ &\ll_{r_0} r p^{(n_1+n_2-1)r}\max_{(i,j)\in I(r)}p^{r-id_1-jd_2-i-j}\\ 
&\ll_{p,r_0} r p^{(n_1+n_2-1)r}p^{-r/(d_1d_2)}.
\end{align*}
Hence we obtain
\begin{align*}
N(r)&= \sum_{i\geq 0}\sum_{\substack{j\geq 0\\ r-id_1-jd_2\geq r_0}} p^{n_1(id_1+jd_2-i)+n_2(id_1+jd_2-j)}N^*(r-id_1-jd_2)\\ &+O( r p^{(n_1+n_2-1)r}p^{-r/(d_1d_2)}).
\end{align*}
Since the summation is restricted to $r_0\leq r-id_1-jd_2$ one has by Lemma \ref{peylem1b}
\begin{equation*}
N^*(r-id_1-jd_2)= p^{-r(n_1+n_2-1)}N^*(r) p^{(r-id_1-jd_2)(n_1+n_2-1)}.
\end{equation*}
Therefore we obtain
\begin{equation*}
N(r) = \sum_{r_0+id_1+jd_2\leq r} p^{-in_1-jn_2+id_1+jd_2}N^*(r) +O( r p^{(n_1+n_2-1)r}p^{-r/(d_1d_2)}).
\end{equation*}
This implies that
\begin{equation*}
\lim_{r\rightarrow \infty}p^{-r(n_1+n_2-1)}N(r)= (1-p^{-(n_1-d_1)})^{-1} (1-p^{-(n_2-d_2)})^{-1} \lim_{r\rightarrow \infty} p^{-r(n_1+n_2-1)}N^*(r) ,
\end{equation*}
which proves the lemma.
\end{proof}

\begin{proof}[Proof of Lemma \ref{peylem1}]
First we note that Lemma \ref{peylem1b} and \ref{peylem1c} imply that
\begin{equation*}
\int_{\{(\bfx;\bfy)\in \Z_p^{n_1+n_2}: F(\bfx;\bfy)=0\}}\ome_{L,p} = \sig_p.
\end{equation*}
The Lemma now follows from this equality and Lemma \ref{peylem1a}.
\end{proof}

Finally we give a proof of Lemma \ref{peylem2}. This is only a slight modification of Proposition VI.5.30 in \cite{Jahnel} to the biprojective setting.

\begin{proof}[Proof of Lemma \ref{peylem2}]
By equation (10) in section 6 in \cite{Bir1961} one has
\begin{equation*}
\sig_\infty = \int_{\substack{W\cap\{ \max_{1\leq i\leq n_1}|x_i|\leq 1, \\ \max_{1\leq j\leq n_2}|y_j|\leq 1\}}} \ome_{L,\infty}.
\end{equation*}
Since the question of the lemma is hence local, it suffices to consider a
subset $V\subset X(\R)$, open in the real topology, such that $V$ is contained
in $x_1y_1\neq 0$ and such that the coordinates $\left(\frac{x_2}{x_1},\ldots, \frac{x_{n_1}}{x_1},\frac{y_2}{y_1},\ldots, \frac{y_{n_2-1}}{y_1}\right)$ define a diffeomorphism $\rho$ with $\rho (V) \subset \A_\R^{n_1+n_2-3}$. Then we set
\begin{equation*}
\sig_\infty (V)= \int_{\substack{\pi^{-1}(V)\cap\{ \max_{1\leq i\leq
      n_1}|x_i|\leq 1, \\ \max_{1\leq j\leq n_2}|y_j|\leq 1 \}}} \ome_{L,\infty}.
\end{equation*}
Using the explicit description of the Leray measure at the beginning of this subsection, we obtain
\begin{equation*}
\sig_\infty (V)= \int_{\substack{\pi^{-1}(V)\cap\{ \max_{1\leq i\leq n_1}|x_i|\leq 1, \\ \max_{1\leq j\leq n_2}|y_j|\leq 1\}}} \frac{\d x_1\ldots \d x_{n_1}\d y_1 \ldots \d y_{n_2-1}}{\left|\frac{\partial F}{\partial y_{n_2}}(\bfx;\bfy)\right|}.
\end{equation*}
We note that the condition $\max_{1\leq i\leq n_1}|x_i|\leq 1$ is equivalent to saying that $|x_1|\leq \left(\max_{1\leq i\leq n_1}\left|\frac{x_i}{x_1}\right|\right)^{-1}$. In the above integral we apply the substitution $x_i=x_1u_i$ for $2\leq i\leq n_1$ and $y_j=y_1v_j$ for $2\leq j\leq n_2-1$. Recall the notation $\bfu= (1,u_2,\ldots, u_{n_1})$ and $\bfv = (1,v_2,\ldots, v_{n_2})$. Then we obtain
\begin{align*}
\sig_\infty (V)= \int |x_1|^{n_1-1-d_1}|y_1|^{n_2-2-(d_2-1)}\frac{\d x_1\d y_1 \d u_2 \ldots \d v_{n_2-1}}{\left| \frac{\partial F}{\partial y_{n_2}}(\bfu;\bfv)\right|},
\end{align*}
with $\pi^{-1}(V)\cap\{ |x_1|^{n_1-d_1}\leq
    h_\infty^1(\bfu)^{-1},\ |y_1|^{n_2-d_2}\leq h_\infty^2(\bfv)^{-1} \}$ as domain of integration. We can rewrite this as
\begin{align*}
  \sig_\infty (V)&= \int_V \frac{2}{n_1-d_1}h_\infty^1(\bfu)^{-1}\frac{2}{n_2-d_2}h_\infty^2(\bfv)^{-1} \frac{ \d u_2 \ldots \d v_{n_2-1}}{\left| \frac{\partial F}{\partial y_{n_2}}(\bfu;\bfv)\right|} \\ &= \frac{4}{(n_1-d_1)(n_2-d_2)}\int_V\ome_\infty,
\end{align*}
which proves our lemma.
\end{proof}

\section{Statement of circle method ingredients}
The strategy for the proof of Theorem \ref{thm3} is as follows. We first count
integral points on the affine cone $W$ given by $F_i(\bfx;\bfy)=0$ for $1\leq
i\leq R$, with $\bfx$ and $\bfy$ restricted to boxes. For this let $\calB_1$
and $\calB_2$ be two boxes in affine $n_1$- and $n_2$-space, and $P_1$ and
$P_2$ be two real parameters larger than $2$. We aim for proving
asymptotic formulas for the number of integer points on $W$ with $\bfx \in
P_1\calB_1$ and $ \bfy \in P_2\calB_2$, possibly restricting our counting functions to
appropriate open subsets of $W$. We will obtain an asymptotic formula, which
holds for all $P_1,P_2\geq 2$, with an error term that saves a small power of
$\min (P_1,P_2)$.\par
We use different approaches depending on the
relative size of $P_1$ and $P_2$. If $P_1$ and $P_2$ are roughly of the same
size or a bounded power of one another, then we import previous work of the
author \cite{bihomforms} which uses a circle method analysis of the type used
in Birch's work \cite{Bir1961}.\par
If $P_2$ is small compared to $P_1$, which means in our setting a small power
of $P_1$, then we take a fibre-wise counting approach. That is, we fix $\bfy$,
for which the resulting variety is not too singular, and count the number of
integer points $\bfx$ of bounded height on the resulting system of
equations. We then add up all the contributions for $\bfy$ in a box of side
lengths $P_2$. In contrast to the case where $P_1$ and $P_2$ are of roughly the
same size, it is here important to exclude bad choices of $\bfy$ as the
example following Theorem \ref{thm3} shows.\par
Theorem \ref{thm2} below is the result of combining both approaches. Together
with asymptotic formulas for the number of integral points on fibers, this is
the main ingredient which is needed to apply a recently developed technique by
Blomer and Br\"udern \cite{BloBru13}. This is carried out in section 9 and will lead to the
proof of Theorem \ref{thm3}.\par

For the following let $P_1,P_2\geq 2$, and define $u \geq 0$ by $u= \frac{\log P_2}{\log P_1}$. We think most of the time of $P_2$ as relatively small compared to $P_1$,
i.e. $u<1$. For fixed $\bfy$ let $N_\bfy (P_1)$ be the number of integer vectors $\bfx$ in $P_1\calB_1$ such that the system of equations (\ref{eqn0}) holds.\par
Since we might like to exclude some fibres for $\bfy$ later, we assume that we are given a set $\calA_1 (\Z)\subset \Z^{n_2}$, and define the counting function
\begin{equation}\label{eqn3}
N_1(P_1,P_2)=\sum_{\bfy \in P_2\calB_2\cap \calA_1(\Z)} N_\bfy (P_1).
\end{equation}
For fixed $\bfy$ and some $\bfalp \in \R^R$ we define the exponential sum
\begin{equation*}
S_\bfy(\bfalp)= \sum_{\bfx \in P_1\calB_1} e\left(\sum_{i=1}^R \alp_i F_i(\bfx;\bfy)\right),
\end{equation*}
where we understand here and later the sum to be over all integer vectors in the given range. Then we have
\begin{equation*}
N_\bfy(P_1)= \int_{[0,1]^R} S_\bfy(\bfalp)\d\bfalp.
\end{equation*}

For fixed $\bfy$ let $V_{1,\bfy}^*$ be the variety in affine $n_1$-space given
by 
\begin{equation*}
\rank \left( \frac{\partial F_i(\bfx;\bfy)}{\partial
    x_j}\right)_{\substack{1\leq i\leq R\\ 1\leq j\leq n_1}}<R,
\end{equation*}
and define $V_{2,\bfx}^*$ analogously.

\begin{theorem}\label{thm1}
For some positive integer $\lam$ let the set $\calA_1(\Z)$ be given by 
\begin{equation*}
\calA_1(\Z) = \{\bfy\in \Z^{n_2}: \dim V_{1,\bfy}^*< \dim V_1^* - n_2+\lam\}.
\end{equation*}
Let $d_1\geq 2$ and $\del >0$, and let $P_1$ and $P_2$ be two real numbers larger than one. Assume that the quantity $u=\frac{\log P_2}{\log P_1}$ satisfies $u d_2(2R^2+3R)+\del<1$, i.e. in particular we have $P_2\leq P_1$. Furthermore, define $K_1$ by  
\begin{equation}\label{defK}
2^{d_1-1}K_1 = n_1+n_2-\dim V_1^*-\lam,
\end{equation}
and write
\begin{equation*}
g_1(u,\del)= (1-ud_2(2R^2+3R)-\del)^{-1}(2R+3)R(d_1-1)(ud_2 R(2R+1)+2\del).
\end{equation*}
Assume that we have
\begin{equation}\label{eqnthm1}
(K_1-R(R+1)(d_1-1)) > g_1(u,\del).
\end{equation}
Then, for $P_1^{\frac{1-\del-(2R+3)Rd_2u}{(2R+3)R(d_1-1)}}>C_3$, one has
\begin{equation*}
N_1(P_1,P_2)= P_1^{n_1-Rd_1}\sum_{\bfy \in P_2\calB_2\cap \calA_1(\Z)} \grS_\bfy J_\bfy + O(P_1^{n_1-Rd_1-\del}P_2^{n_2-Rd_2}),
\end{equation*}
where $\grS_\bfy$ and $J_\bfy$ are given in Lemma \ref{lem4.3} and
Lemma \ref{lem4.4}. The complement $\calA_1^c(\Z)$ of the set $\calA_1(\Z)$ can be
given as the set of zeros of a system of homogeneous polynomials in
$\bfy$.\par

\end{theorem}

This theorem is useful when $P_2$ is relatively small compared to $P_1$. We write out the same theorem, where the roles of $\bfx$ and $\bfy$ are reversed.

\begin{theorem}\label{thm1b}
Let $d_2\geq 2$, and $\del >0$. Assume that we have $ d_1(2R^2+3R)+\del u<u$. For some positive integer $\lam_2$ let the set $\calA_2(\Z)$ be given by 
\begin{equation*}
\calA_2 (\Z)= \{\bfx\in \Z^{n_1}: \dim V_{2,\bfx}^*< \dim V_2^* - n_1+\lam_2\}.
\end{equation*}
Define the counting function $N_2(P_1,P_2)$ by
\begin{equation*}
N_2(P_1,P_2)=\sharp \{ \bfx\in \calA_2(\Z) \cap P_1\calB_1,\ \bfy\in P_2\calB_2\cap \Z^{n_2}: F_i(\bfx;\bfy)=0,\ 1\leq i\leq R\}.
\end{equation*}
Furthermore, define $K_2$ by  
\begin{equation}\label{defK}
2^{d_2-1}K_2 = n_1+n_2-\dim V_2^*-\lam_2,
\end{equation}
and write
\begin{equation*}
g_2(u,\del)=(u-d_1(2R^2+3R)-u\del)^{-1}(2R+3)R(d_2-1)(d_1 R(2R+1)+2u\del).
\end{equation*}
Assume that we have
\begin{equation*}\label{eqnthm1}
(K_2-R(R+1)(d_2-1)) > g_2(u,\del).
\end{equation*}
Then, for $P_2^{\frac{u-u\del-(2R+3)Rd_1}{u(2R+3)r(d_2-1)}}>C_3$, we have
\begin{equation*}
N_2(P_1,P_2)= P_2^{n_2-Rd_2}\sum_{\bfx \in P_1\calB_1\cap \calA_2(\Z)} \grS_\bfx J_\bfx + O(P_2^{n_2-Rd_2-\del}P_1^{n_1-Rd_1}),
\end{equation*}
where $\grS_\bfx$ and $J_\bfx$ are defined analogously as $\grS_\bfy$ and $J_\bfy$. As in Theorem \ref{thm1}, the complement $\calA_2^c(\Z)$ of the set $\calA_2(\Z)$ is given as the set of zeros of a system of homogeneous polynomials in $\bfx$. 
\end{theorem}

The proof of Theorem \ref{thm1} and Theorem \ref{thm1b} is carried out in the next four sections. We first seek asymptotic formulas for the counting functions $N_\bfy( P_1)$ and then essentially add up the contributions as in equation (\ref{eqn3}).\par
Next we repeat a result for counting solutions to the system of equations
(\ref{eqn0}) in a situation where $P_1$ and $P_2$ are of similar size. This result was proved in \cite{bihomforms}, and we repeat it here, since we use is for the proof of Theorem \ref{thm2} below. For this we introduce the counting function $N'(P_1,P_2)$ to be the number of integer vectors $\bfx \in P_1\calB_1$ and $\bfy\in P_2\calB_2$ such that $F_i(\bfx;\bfy)=0$, for $1\leq i\leq R$. 

\begin{theorem}\label{thmbihom}
Assume $u\leq 1$ and $\min \{n_1,n_2\}>R$ and suppose that we have
\begin{equation*}
n_1+n_2-\dim V_i^* >2^{d_1+d_2-2}\max \{ R(R+1)(d_1+d_2-1),R(d_1/u+d_2)\},
\end{equation*}
for $i=1,2$. Then we have the asymptotic formula 
\begin{equation*}
N'(P_1,P_2)=\sig P_1^{n_1-Rd_1}P_2^{n_2-Rd_2}+O(P_1^{n_1-Rd_1-\deltil}P_2^{n_2-Rd_2}),
\end{equation*}
for some real number $\sig$ and some $\deltil >0$. Here $\sig$ is as usual the product of a singular series $\grS$ and singular integral $J$ (taken with respect to the box $[-1,1]^{n_1+n_2}$), which are for example defined in Schmidt's work \cite{Schmidt1985}, equation 3.10. Furthermore, the constant $\sig$ is positive if\par
i) the $F_i(\bfx;\bfy)$ have a common non-singular $p$-adic zero for all $p$,\par
ii) the $F_i(\bfx;\bfy)$ have a non-singular real zero in the box $\calB_1\times \calB_2$ and $\dim V(0)= n_1+n_2-R$, where $V(0)$ is the affine variety given by the system of equations (\ref{eqn0}).
\end{theorem}

Assume for the following that $d_1+d_2 > 2$, and fix some small $\del>0$. For a real number $t$, write $\lceil t\rceil$ for the smallest integer larger than or equal to $t$.\par
Now let $b_1 > d_2 (2R^2+3R)$ be the solution to the quadratic equation
\begin{align*}
2^{d_1+d_2-2} R(b_1d_1+d_2) =&2^{d_1-1}(g_1(1/b_1,\del)+R(R+1)(d_1-1))\\ &+\lceil R(b_1d_1+d_2)+\del\rceil.
\end{align*}
Note that $g_1(u,\del)$ is monoton growing on $ud_2(2R^2+3R)+\del <1$. In
considering the value $b=2d_2(2R^2+3R)$, a short calculation shows that
\begin{equation*}
2^{d_1+d_2-2} R(b_1d_1+d_2) \leq 3 \cdot 2^{d_1+d_2}R^3 d_1d_2,
\end{equation*}
for $\del$ sufficiently small.

Next we set $u_1=1/b_1$. Our goal is to find an asymptotic formula for a modified form of the counting function $N'(P_1,P_2)$, which holds for all values of $P_1,P_2\geq 1$. For values of $0< u\leq u_1$ we will use Theorem \ref{thm1} above. In the range $u_1 < u\leq 1$ we use Theorem \ref{thmbihom}.

The above theorems essentially cover the case of $P_2\leq P_1$. To obtain asymptotic formulas for $P_2 > P_1$ we interchange the roles of $\bfx$ and $\bfy$. Thus, we define analogously to $b_1$ the real number $b_2$ to be the solution of the quadratic equation 
\begin{align*}
2^{d_1+d_2-2} R(b_2d_2+d_1)  =&2^{d_2-1}(g_2(b_2,\del)+R(R+1)(d_2-1))\\ &+\lceil R(b_2d_2+d_1)+\del\rceil. 
\end{align*}

Next set $\lam_1 =\lceil R(b_1d_1+d_2)+\del \rceil$ and $\lam_2 =\lceil
R(b_2d_2+d_1)+\del \rceil$. Consider the open subsets $U_1=\calA_2$ and
$U_2=\calA_1$, and their product $U=U_1\times U_2 \subset
\A_\C^{n_1+n_2}$. Then we define the counting function $N_U(P_1,P_2)$ to be
the number of integer vectors $\bfx\in P_1\calB_1$ and $\bfy\in P_2\calB_2$
with $(\bfx;\bfy)\in U$ such that the system of equations (\ref{eqn0})
holds. We set
\begin{equation*}
\phi (d_1,d_2,R)= 2^{d_1+d_2-2} R \max\{ (b_1d_1+d_2),
(b_2d_2+d_1)\}.
\end{equation*}

\begin{theorem}\label{thm2}
Assume that $d_1,d_2\geq 2$ and $n_1,n_2 >R$, and that 
\begin{equation}\label{eqnthm2}
n_1+n_2 -\max \{\dim V_1^*,\dim V_2^*\} > \phi (d_1,d_2,R).
\end{equation}
Then we have
\begin{equation*}
N_U(P_1,P_2)= \sig P_1^{n_1-Rd_1}P_2^{n_2-Rd_2} +O
(P_1^{n_1-Rd_1}P_2^{n_2-Rd_2} \min \{P_1,P_2\}^{-\deltil}),
\end{equation*}
for some $\deltil >0$ and positive real numbers $P_1\geq 2$ and $P_2\geq
2$. Here $\sig$ is the same constant as in Theorem \ref{thmbihom}. Moreover,
we have 
\begin{equation*}
\phi (d_1,d_2,R)\leq 3 \cdot 2^{d_1+d_2}d_1d_2R^3.
\end{equation*}
\end{theorem}

This is the precurser of Theorem \ref{thm3}. There are mainly two steps left from here to prove Theorem \ref{thm3}. On the one hand, we have to replace the height function $\max_i|x_i|\leq P_1$ and $\max_j |y_j|\leq P_2$ by the anticanonical height function given in the introductory section. This is done using techniques developed by Blomer and Br\"udern \cite{BloBru13}. On the other hand, we still count all integer points on the affine cone of an open subset of $X$. We will perform a M\"obius inversion to obtain results on the counting function in biprojective space. 

\section{Exponential sums}
Our first goal is to establish a form of Weyl-lemma for the exponential sum $S_\bfy(\bfalp)$. Write $\bfxtil= (\bfx^{(1)},\ldots, \bfx^{(d_1)})$, and let $\Gam_\bfy(\bfxtil;\bfalp)$ be the multilinear form, which is associated to
\begin{equation*}
d_2! \sum_{i=1}^R \alp_i F_i(\bfx;\bfy),
\end{equation*}
for fixed $\bfy$. Write $\bfe_j$ for the $j$th unit vector. By Lemma 2.1 of Birch's paper \cite{Bir1961} we have the estimate
\begin{equation*}
|S_\bfy(\bfalp)|^{2^{d_1-1}}\ll P_1^{(2^{d_1-1}-d_1)n_1}\sum \left(\prod_{j=1}^{n_1} \min (P_1, \Vert \Gam_\bfy (\bfe_j,\bfx^{(2)},\ldots, \bfx^{(d_1)};\bfalp)\Vert^{-1}))\right),
\end{equation*}
where $\sum$ is over all integer vectors $\bfx^{(2)},\ldots, \bfx^{(d_1)}\in P_1\calE$, where $\calE$ is the $n_1$-dimensional unit cube. Let $L_\bfy(P,P^{-\eta},\bfalp)$ be the number of such integer vectors in $P\calE$ such that
\begin{equation*}
\Vert \Gam_\bfy(\bfe_j,\bfx^{(2)},\ldots, \bfx^{(d_1)};\bfalp)\Vert < P^{-\eta},
\end{equation*}
for all $1\leq j\leq n_1$. Then, again by \cite{Bir1961}, Lemma 2.4, we have the following result.

\begin{lemma}
Let $P$ and $\kap$ be some real parameters. If $|S_\bfy(\bfalp)|> P_1^{n_1+\eps}P^{-\kap}$, then one has
\begin{equation*}
L_\bfy (P_1^\tet, P_1^{-d_1+(d_1-1)\tet},\bfalp) \gg P_1^{(d_1-1)n_1\tet}P^{-2^{d_1-1}\kap},
\end{equation*}
for fixed $0 < \tet\leq 1$ and any $\eps >0$.
\end{lemma}

Next define the multilinear forms $\Gam_\bfy^{(i)}(\bfxtil)$ for $1\leq i\leq R$ in such a way that
\begin{equation*}
\Gam_\bfy(\bfxtil;\bfalp)=\sum_{i=1}^R \alp_i\Gam_\bfy^{(i)}(\bfxtil),
\end{equation*}
for all real vectors $\bfalp$. Write $\bfxhat = (\bfx^{(2)},\ldots, \bfx^{(d_1)})$. Suppose that we are given some $\bfxhat\in (-P_1^\tet,P_1^\tet)^{n_1(d_1-1)}$ such that the matrix
\begin{equation*}
(\Gam_\bfy^{(i)} (\bfe_j,\bfxhat))_{i,j}
\end{equation*}
has full rank. For convenience we assume that the leading $R\times R$ minor has full rank. For all $1\leq l\leq n_1$, we can write
\begin{equation*}
\Gam_\bfy (\bfe_l, \bfxhat;\bfalp)=\atil_l+\deltil_l,
\end{equation*}
for some integers $\atil_l$ and real $\deltil_l$ with $|\deltil_l|< P_1^{-d_1+(d_1-1)\tet}$. Furthermore, let
\begin{equation*}
q=|\det (\Gam_\bfy^{(i)}(\bfe_j,\bfxhat))_{1\leq i,j\leq R}|.
\end{equation*}
Now we consider the system of linear equations
\begin{equation*}
\sum_{i=1}^R \alp_i \Gam_\bfy^{(i)}(\bfe_j,\bfxhat)= \atil_j+\deltil_j,
\end{equation*}
for $1\leq j\leq R$. We want to solve this in $\alp_i$. For this let $A_\bfy(\bfxhat)$ be the inverse matrix of $(\Gam_\bfy^{(i)}(\bfe_j,\bfxhat))_{1\leq i,j\leq R}$. We note that $qA_\bfy(\bfxhat)$ has integer entries which are essentially given by certain submatrices of $(\Gam_\bfy^{(i)}(\bfe_j,\bfxhat))$. Now we have
\begin{equation*}
\alp_i=\sum_{j=1}^R A_\bfy(\bfxhat)_{i,j} (\atil_j+\deltil_j),
\end{equation*}
for all $1\leq i\leq R$, where we write 
\begin{equation*}
A_\bfy(\bfxhat)= (A_\bfy(\bfxhat)_{i,j})_{1\leq i,j\leq R}.
\end{equation*}
We set $a_i=q\sum_{j=1}^R A_\bfy(\bfxhat)_{i,j}\atil_j$ and obtain then the approximation
\begin{equation*}
|q\alp_i-a_i| \leq q \left| \sum_{j=1}^R A_\bfy(\bfxhat)_{i,j}\deltil_j\right|,
\end{equation*}
for all $1\leq i\leq R$. This proves the following lemma.

\begin{lemma}\label{weyl1}
Let $P$ and $\kap$ be some real parameters and $ 0< \tet \leq 1$ be fixed. Then one of the following alternatives holds.\\
i) One has the bound $|S_\bfy(\bfalp)|< P_1^{n_1+\eps}P^{-\kap}$.\\
ii) There exist integers $1\leq q\leq P_1^{R\tet (d_1-1)}|\bfy|^{Rd_2}$ and $a_i$ for $1\leq i\leq R$ with $\gcd (q,a_1,\ldots, a_R)=1$ such that
\begin{equation*}
2|q\alp_i-a_i| \leq P_1^{-d_1+R\tet (d_1-1)}|\bfy|^{(R-1)d_2},
\end{equation*}
for all $1\leq i\leq R$. Here were write $|\bfy|$ for the maximums norm $|\bfy| =\max_{i}|y_i|$.\\
iii) The number of integer vectors $\bfxhat \in (-P_1^\tet, P_1^\tet)^{n_1(d_1-1)}$ such that 
\begin{equation}\label{defMy}
\rank (\Gam_\bfy^{(i)}(\bfe_l,\bfxhat))<R
\end{equation}
is bounded below by
\begin{equation*}
\geq C_1 P_1^{\tet_1n_1(d_1-1)}P^{-2^{d_1-1}\kap},
\end{equation*}
for some positive constant $C_1$.
\end{lemma}
Our next goal is to show that we can omit alternative iii) in the above lemma
for certain choices of $\bfy$ and a suitable dependence of $\kap$ and
$\tet$. 
\begin{com}
For this let $V_1^*$ be the variety in complex affine $n_1+n_2$-space given by 
\begin{equation}\label{defsingloc}
\rank \left( \frac{\partial F_i(\bfx;\bfy)}{\partial x_j}\right)_{\substack{
    1\leq i\leq R\\ 1\leq j\leq n_1}}<R.
\end{equation}
For some $\bfz\in \C^{n_2}$ we define $V_{1,\bfz}^*$ as the intersection
\begin{equation*}
V_{1,\bfz}^*= V_1^*\cap \{y_1=z_1\}\cap \ldots \cap \{y_{n_2}=z_{n_2}\}.
\end{equation*}
Analogously, let $V_2^*\subset \A_\C^{n_1+n_2}$ be the variety given by 
\begin{equation}
\rank \left( \frac{\partial F_i(\bfx;\bfy)}{\partial y_j}\right)_{\substack{
    1\leq i\leq R\\ 1\leq j\leq n_2}}<R.
\end{equation}
\end{com}
Recall that we have defined
\begin{equation*}
\calA_1 = \{\bfz\in \A_\C^{n_2}: \dim V_{1,\bfz}^*< \dim V_1^* - n_2+\lam\},
\end{equation*}
for some integer parameter $\lam$ to be chosen later.\par
Assume now that we are
given some $\bfy\in \calA_1 (\Z)$ such that alternative iii) of Lemma \ref{weyl1}
holds with $P=P_1$ and $\kap = K_1 \tet$, where $K_1$ is defined as in Theorem \ref{thm1}, i.e. 
\begin{equation}\label{defK}
2^{d_1-1}K_1 = n_1+n_2-\dim V_1^*-\lam.
\end{equation}
Furthermore, let $\calM_\bfy \subset \A_\C^{n_1(d_1-1)}$ be the affine variety
given by (\ref{defMy}), and define $M_\bfy(P_1^\tet)$ to be the number of integer points
$\bfxhat$ on $\calM_\bfy$ with $\bfxhat \in
(-P_1^\tet,P_1^\tet)^{n_1(d_1-1)}$. We note that the degree of $\calM_\bfy$ is
bounded independently of $\bfy$. Thus, the proof of Theorem 3.1 in \cite{Brown}
delivers
\begin{equation*}
M_\bfy(P_1^\tet)\ll P_1^{\tet \dim \calM_\bfy},
\end{equation*}
for some implied constant which is independent of $\bfy$.\par 

Next consider in $\A_\C^{n_1(d_1-1)}$ the diagonal $\calD$ given by
$\bfx^{(2)}=\ldots = \bfx^{(d_1)}$. Then $\calM_\bfy \cap \calD$ is isomorphic
to $V_{1,\bfy}^*$ and we have 
\begin{equation*}
\dim \calM_\bfy\cap \calD \geq \dim \calM_\bfy +\dim \calD - n_1(d_1-1),
\end{equation*}
and hence
\begin{equation*}
\dim \calM_\bfy \leq n_1(d_1-2) +\dim V_{1,\bfy}^*.
\end{equation*}
We conclude that there exists a positive constant $C_2$, independent of
$\bfy$, such that for all $\bfy\in \calA_1 (\Z)$ we have
\begin{equation*}
M_\bfy(P_1^\tet) < C_2 P_1^{\tet(n_1(d_1-2)+\dim V_1^*-n_2+\lam-1)}.
\end{equation*}
If alternative iii) of Lemma \ref{weyl1} holds, then we have
\begin{equation*}
C_1 P_1^{\tet(n_1(d_1-1)-2^{d_1-1}K_1)} < C_2 P_1^{\tet(n_1(d_1-2)+\dim
  V_1^*-n_2+\lam-1)},
\end{equation*}
which is equivalent to 
\begin{equation*}
C_1P_1^\tet < C_2,
\end{equation*}
by definition of $K_1$. We have now established the following lemma.

\begin{lemma}\label{weyl2}
There is a positive constant $C_3$ such that the following holds. Let $0< \tet
\leq 1$ and $P_1\geq 1$ with $P_1^\tet >C_3$, and assume that $\bfy \in
\calA_1 (\Z)$. Then we have either the bound
\begin{equation*}
|S_\bfy (\bfalp)| < P_1^{n_1-K_1\tet +\eps},
\end{equation*}
or alternative ii) of Lemma \ref{weyl1} holds.
\end{lemma}

Next we give an estimate for the number of integer vectors of bounded height
which are not in $\calA_1$.

\begin{lemma}\label{lem2.4}
Denote by $\calA_1^c$ the complement of $\calA_1$. Then we have
\begin{equation*}
\sharp \{\bfz\in (-P_2,P_2)^{n_2}\cap \calA_1^c (\Z)\}\ll P_2^{n_2-\lam}.
\end{equation*}
Furthermore, the set of all vectors $\bfz$ with 
\begin{equation*}
\dim V_{1,\bfz}^* \geq \dim V_1^* -n_2+\lam
\end{equation*}
is a Zariski-closed subset of $\A_\C^{n_2}$.
\end{lemma}

\begin{proof}
First we show that
\begin{equation*}
\calA_1^c = \{\bfz\in \A_\C^{n_2}: \dim V_{1,\bfz}^* \geq \dim V_1^*-n_2+\lam\}
\end{equation*}
is a closed subset in $\A_\C^{n_2}$. For this let $\Del_1,\ldots, \Del_r$ be
all the $R\times R$-subdeterminants of the matrix $(\partial F_i(\bfx;\bfy)/\partial x_j)_{1\leq i\leq R,1\leq j\leq n_1}$. They define a closed subset $Y$ of $\P_\C^{n_1-1}\times \A_\C^{n_2}$. We note that the morphism
\begin{equation*}
\pi: Y \hookrightarrow \P_\C^{n_1-1}\times \A_\C^{n_2}\rightarrow \A_\C^{n_2}
\end{equation*}
is projective and hence closed. Thus, we can apply Corollaire 13.1.5 from
\cite{EGA4} and see that 
\begin{equation*}
\{\bfz\in \A_\C^{n_2}: \dim Y_\bfz \geq \dim V_1^*-n_2+\lam-1\}
\end{equation*}
is closed, and hence $\calA_1^c$ is closed, since $\dim Y_\bfz +1= \dim
V_{1,\bfz}^*$.\par
Next we note that the intersection $Y\cap (\P_\C^{n_1-1}\times \calA_1^c)$ is given by the disjoint product of the fibres $\cup_{\bfz\in \calA_1^c}\pi^{-1}(\bfz)$. If $\dim V_1^* - n_2+\lam -1 \geq 0$, then all the fibers $\pi^{-1}(\bfz)$ are nonempty for $\bfz \in \calA_1^c$. Hence, we have
\begin{equation*}
\dim \calA_1^c +\dim V_1^*-n_2+\lam -1\leq \dim Y= \dim V_1^*-1,
\end{equation*}
which implies
\begin{equation*}
\dim \calA_1^c \leq n_2-\lam.
\end{equation*}
If $\dim V_1^* - n_2+\lam \leq 0$, then the first part of the lemma is trivial since $n_2\leq \dim V_1^*$.\par
This delivers the required bound on integer points on $\calA_1^c$.
\end{proof}

\section{Circle method}

Throughout this section we assume that $d_1\geq 2$.\par
For some $ 0< \tet \leq 1$ and $\bfy\in \Z^{n_2}$, we define the major arc $\grM_{\bfa, q}^{\bfy}(\tet)$ to
be the set of $\bfalp \in [0,1]^R$ such that
\begin{equation*}
2|q\alp_i-a_i| \leq P_1^{-d_1+R\tet (d_1-1)}|\bfy|^{(R-1)d_2},
\end{equation*}
and set
\begin{equation*}
\grM^{\bfy} (\tet)= \bigcup_{q\leq P_1^{R\tet (d_1-1)}|\bfy|^{Rd_2}}\bigcup_{\bfa}
\grM_{\bfa,q}^{\bfy}(\tet),
\end{equation*}
where the second union is over all integers $0\leq a_1,\ldots, a_R < q$ such that $\gcd
(q,a_1,\ldots, a_R)=1$. Let the minor arcs $\grm^{\bfy} (\tet)$ be the complement of $\grM^{\bfy}(\tet)$ in
$[0,1]^R$. We also define the slightly larger major arcs
$\grM_{\bfa, q}^{'\bfy}(\tet)$ by
\begin{equation*}
2|q\alp_i-a_i| \leq qP_1^{-d_1+R\tet (d_1-1)}|\bfy|^{(R-1)d_2},
\end{equation*}
and let $\grM^{'\bfy}(\tet)$ be defined in an analogous way as $\grM^{\bfy} (\tet)$. In the next
lemma we show that the major arcs $\grM^{'\bfy}_{\bfa,q} (\tet)$ are disjoint for sufficiently
small $\tet$, depending on $|\bfy|$.

\begin{lemma}\label{lem3.1}
Assume that
\begin{equation}\label{MAdisj}
P_1^{-d_1+3R\tet (d_1-1)}|\bfy|^{(3R-1)d_2}<1.
\end{equation}
Then the major arcs $\grMyd_{\bfa,q}(\tet)$ are
disjoint.
\end{lemma}

\begin{proof}
Assume that we are given some $\bfalp \in \grMyd_{\bfa, q}(\tet)\cap
\grMyd_{\bfatil,\qtil}(\tet)$ with both $q,\qtil\leq P_1^{R\tet
  (d_1-1)}|\bfy|^{Rd_2}$. Then we have some $1\leq i\leq R$ with
\begin{equation*}
\frac{1}{q\qtil}\leq \left| \frac{a_i}{q}-\frac{\atil_i}{\qtil}\right| \leq
P_1^{-d_1+ R\tet (d_1-1)}|\bfy|^{(R-1)d_2}.
\end{equation*}
This implies
\begin{equation*}
1\leq P_1^{-d_1+3R\tet (d_1-1)}|\bfy|^{(3R-1)d_2},
\end{equation*}
which is a contradiction to our assumption (\ref{MAdisj}).
\end{proof}

The next lemma reduces our counting issue to a major arc situation.

\begin{lemma}\label{lem3.2}
Let $\bfy\in \calA_1(\Z)$, and $P_1^\tet >C_3$. Assume that (\ref{MAdisj}) holds, and that we have
\begin{equation}\label{minarc}
K_1>(d_1-1)R(R+1).
\end{equation}
Let $\phi (\bfy)= P_1^{R\tet (d_1-1)}|\bfy|^{Rd_2}$, and define 
\begin{equation*}
\Del (\tet, K_1)= \tet (K_1-(d_1-1)R(R+1)).
\end{equation*}
Then we have the asymptotic formula
\begin{equation*}
N_\bfy(P_1)= \sum_{q\leq \phi (\bfy)}\sum_\bfa \int_{\grMyd_{\bfa,q}(\tet)}S_\bfy(\bfalp)\d\bfalp + O( P_1^{n_1-Rd_1-\Del (\tet,K_1)+\eps}|\bfy|^{R^2d_2}),
\end{equation*}
where the summation over
$\bfa$ is over all $0\leq a_i <q$ with $\gcd (q,a_1,\ldots, a_R)=1$.
\end{lemma}

\begin{proof}
By Lemma \ref{lem3.1} the major arcs $\grMyd(\tet)$ are disjoint for $\tet$ as in the assumptions. Hence we can write
\begin{equation*}
N_\bfy(P_1)= \sum_{1\leq q\leq \phi(\bfy)} \sum_\bfa
\int_{\grMyd_{\bfa,q}(\tet)} S_{\bfy}(\bfalp)\d\bfalp + \calE(\bfy),
\end{equation*}
with a minor arc contribution of the form
\begin{equation*}
\calE (\bfy)= \int_{\grm^{\bfy}(\tet)} |S_\bfy(\bfalp)|\d\bfalp.
\end{equation*}
First we shortly
estimate the size of the major arcs $\grMy (\tet)$ by
\begin{align*}
\meas (\grMy (\tet))&\ll \sum_{q\leq \phi(\bfy)} \sum_{\bfa} q^{-R} P_1^{-Rd_1+ R^2\tet
  (d_1-1)}|\bfy|^{R(R-1)d_2} \\ &\ll P_1^{-Rd_1+\tet (d_1-1)R(R+1)}|\bfy|^{R^2d_2}.
\end{align*}
Next we choose a sequences of real numbers $1 = \vartet_T > \vartet_{T-1} >
\ldots > \vartet_1 > \vartet_0 =\tet>0$ with 
\begin{equation}\label{varteti}
\eps > (\vartet_{i+1}-\vartet_i)(d_1-1)R(R+1),
\end{equation}
for some small $\eps>0$. Note that we certainly can achieve this with $T\ll P^{\eps}$.\par
Since $\bfy\in \calA_1 (\Z)$ we can now estimate by Lemma \ref{weyl2} the contribution
on the complement of $\grMy(\vartet_T)$ by
\begin{align*}
\int_{\bfalp\notin \grMy(\vartet_T)}|S_\bfy(\bfalp)|\d\bfalp &\ll
P_1^{n_1-K_1\vartet_T+\eps} \\ &\ll P_1^{n_1-Rd_1-\Del(\tet,K_1)+\eps},
\end{align*}
since
\begin{equation*}
\tet (K_1-(d_1-1)R(R+1))\leq K_1-Rd_1,
\end{equation*}
for $d_1\geq 2$.\par
On the set $\grMy(\vartet_{i+1})\setminus \grMy(\vartet_i)$ for $i=0,\ldots,T-1$
we obtain
\begin{align*}
\int_{\bfalp \in \grMy (\vartet_{i+1})\setminus \grMy  (\vartet_i)}|S_\bfy(\bfalp)|\d\bfalp &\ll \meas (\grMy (\vartet_{i+1}))
P_1^{n_1-K_1\vartet_i+\eps} \\ &\ll P_1^{n_1-Rd_1-K\vartet_i +\eps
  +\vartet_{i+1}(d_1-1)R(R+1)}|\bfy|^{R^2d_2} \\ &\ll
P_1^{n_1-Rd_1-\Del (\tet, K_1)+2\eps}|\bfy|^{R^2d_2},
\end{align*}
since
\begin{equation*}
-K_1\vartet_i  +\vartet_{i+1}(d_1-1)R(R+1)= (\vartet_{i+1}-\vartet_i)(d_1-1)R(R+1)-\Del (\vartet_i,K_1).
\end{equation*}
This shows that
\begin{equation*}
\calE(\bfy)\ll P_1^{n_1-Rd_1+\Del (\tet,K_1)+3\eps}|\bfy|^{R^2d_2},
\end{equation*}
as required.
\end{proof}

\section{Major arcs}

\begin{lemma}\label{MAapprox}\label{lem4.1}
Let $\bfy\in \Z^{n_2}$. Assume that there is some $1\leq q\leq P_1^{R\tet (d_1-1)}|\bfy|^{Rd_2}$ and that
there are integers $a_1,\ldots, a_R$ with
\begin{equation*}
2|q\alp_i-a_i|\leq qP_1^{-d_1+R\tet (d_1-1)}|\bfy|^{(R-1)d_2},
\end{equation*}
for all $1\leq i\leq R$. Write $\bet_i=\alp_i-a_i/q$ for all $i$. Then one has
\begin{equation*}
S_\bfy (\bfalp)=P_1^{n_1}q^{-n_1} S_{\bfa,q}(\bfy) I_\bfy(P_1^{d_1}\bfbet)
+O(P_1^{n_1-1+2R\tet (d_1-1)}|\bfy|^{2Rd_2}),
\end{equation*}
with the exponential sum
\begin{equation*}
S_{\bfa, q}(\bfy)= \sum_{\bfz \mmod q} e\left( \sum_{i=1}^R \frac{a_i}{q}
  F_i(\bfz;\bfy)\right),
\end{equation*}
and the integral
\begin{equation*}
I_\bfy(\bfbet)= \int_{\bfv \in \calB_1} e\left( \sum_i \bet_i
  F_i(\bfv;\bfy)\right) \d\bfv.
\end{equation*}
\end{lemma}

\begin{com}
We note that we would expect the approximation from this lemma to be
non-trivial in the range
\begin{equation*}
P_1^{-1+2R\tet (d_1-1)}P_2^{2Rd_2}\ll P_1^{-\eps},
\end{equation*}
thus for
\begin{equation*}
2R\tet (d_1-1)+2R ud_2 <1.
\end{equation*}
Therefore, we later expect to choose our parameter $u$ in the range
\begin{equation*}
u< (2Rd_2)^{-1}(1-2R \tet (d_1-1)).
\end{equation*}
\end{com}

\begin{proof}
First we write
\begin{equation*}
S_\bfy (\bfalp)= \sum_{\bfz \mmod q} e\left( \sum_i \frac{a_i}{q}
  F_i(\bfz;\bfy)\right) S_3(\bfz),
\end{equation*}
with the sum
\begin{equation*}
S_3(\bfz)= \sum_{\bft} e\left( \sum_i \bet_i F_i(q \bft+\bfz;\bfy)\right),
\end{equation*}
where the summation is over all integer vectors $\bft$ with $q\bft +\bfz \in
P_1\calB_1$. Consider two such vectors $\bft$ and $\bft'$ with
$|\bft-\bft'|\ll 1$ in the maximums norm. Then we have
\begin{equation*}
|F_i(q \bft +\bfz ;\bfy)- F_i(q\bft'+ \bfz; \bfy)|\ll q P_1^{d_1-1}|\bfy|^{d_2},
\end{equation*}
and therefore
\begin{align*}
S_3(\bfz) &= \int_{q\bfvtil \in P_1\calB_1} e\left( \sum_i \bet_i F_i
  (q\bfvtil; \bfy)\right) \d\bfvtil\\ & +O\left(\sum_i |\bet_i| qP_1^{d_1-1}
  |\bfy|^{d_2} \left(\frac{P_1}{q}\right)^{n_1} + \left(\frac{P_1}{q}\right)^{n_1-1}\right).
\end{align*}
After a coordinate transformation we obtain
\begin{align*}
S_3 &= P_1^{n_1}q^{-n_1} \int_{\bfv \in \calB_1} e\left(\sum_i P_1^{d_1}\bet_i
  F_i(\bfv;\bfy)\right) \d\bfv +O(q^{-n_1+1} P_1^{n_1-1 +R\tet
  (d_1-1)}|\bfy|^{Rd_2})\\ = & P_1^{n_1} q^{-n_1} I_\bfy (P_1^{d_1}\bfbet) +O(q^{-n_1+1} P_1^{n_1-1 +R\tet
  (d_1-1)}|\bfy|^{Rd_2}),
\end{align*}
which proves the lemma.
\end{proof}

Now we combine Lemma \ref{lem4.1} with Lemma \ref{lem3.2} and obtain the
following approximation for the counting function $N_\bfy(P_1)$. Let $\phitil
(\bfy)=\tfrac{1}{2}P_1^{R\tet(d_1-1)}|\bfy|^{(R-1)d_2}$.

\begin{lemma}\label{lem4.2}
Set
\begin{equation*}
\eta (\tet) = 1-(3+2R)R\tet (d_1-1).
\end{equation*}
Under the same assumptions as in Lemma \ref{lem3.2} we have
\begin{align*}
N_\bfy(P_1)=& P_1^{n_1-Rd_1} \grS_\bfy(\phi(\bfy))
J_\bfy(\phitil(\bfy))\\ &+O( P_1^{n_1-Rd_1-\Del (\tet,K_1)+\eps}|\bfy|^{R^2d_2} +P_1^{n_1-Rd_1-\eta(\tet)}|\bfy|^{2R(R+1)d_2}),
\end{align*}
with some truncated singular series
\begin{equation*}
\grS_\bfy(\phi(\bfy))= \sum_{q\leq \phi(\bfy)}q^{-n_1}\sum_{\bfa} S_{\bfa,q}(\bfy),
\end{equation*}
where the summation is over all $0\leq a_1,\ldots, a_R<q$ with $\gcd
(a_1,\ldots, a_R,q)=1$. Furthermore the truncated singular integral is given
by
\begin{equation*}
J_\bfy(\phitil(\bfy))= \int_{\bfbet \leq \phitil(\bfy)} I_\bfy (\bfbet)\d\bfbet.
\end{equation*}
\end{lemma}

\begin{proof}
Write $O(E_1)$ for $O(P_1^{n_1-Rd_1-\Del (\tet,K_1)+\eps}|\bfy|^{R^2d_2} )$. An application of Lemma \ref{lem3.2} leads to
\begin{equation*}
N_\bfy(P_1)=  \sum_{q\leq \phi(\bfy)}
\sum_{\bfa}  \int_{\grMyd_{\bfa,q}(\tet)} S_\bfy (\bfalp)\d\bfalp +
O(E_1).
\end{equation*}
We insert the approximation of Lemma \ref{lem4.1} for $S_\bfy (\bfalp)$, and obtain
\begin{equation*}
N_\bfy(P_1)= P_1^{n_1} \sum_{q\leq \phi(\bfy)}q^{-n_1}\sum_\bfa
S_{\bfa,q}(\bfy) \int_{|\bfbet|\leq \phitil (\bfy)P_1^{-d_1}}I_\bfy
(P_1^{d_1}\bfbet)\d\bfbet +O(E_1) +O(E_2),
\end{equation*}
with
\begin{equation*}
E_2= \meas (\grMyd(\tet))P_1^{n_1-1+2R\tet(d_1-1)}|\bfy|^{2Rd_2}.
\end{equation*}
A variable subsitution in the integral over $\bfbet$ shows that we have
already obtained the required main term.\par
We note that
\begin{align*}
\meas (\grMyd (\tet))&\ll \sum_{q\leq \phi(\bfy)}\sum_\bfa
P_1^{-Rd_1}\phitil(\bfy)^R \\ &\ll P_1^{-Rd_1}\phitil(\bfy)^R \phi
(\bfy)^{R+1}.
\end{align*}
Hence, the second error term $E_2$ is
bounded by
\begin{align*}
E_2\ll P_1^{n_1-Rd_1-\eta(\tet)}|\bfy|^{2Rd_2+R(R-1)d_2+(R+1)Rd_2}\ll
P_1^{n_1-Rd_1-\eta(\tet)}|\bfy|^{2R(R+1)d_2},
\end{align*}
with
\begin{align*}
\eta (\tet) &= 1-2R\tet(d_1-1)-(R+1)R\tet(d_1-1)-R^2\tet (d_1-1)\\ &=
1-(3+2R)R\tet (d_1-1).
\end{align*}
\end{proof}

\begin{lemma}\label{lem4.3}
Let $\bfy\in \calA_1(\Z)$, and assume that we have $K_1 > R^2 (d_1-1) +\del$. Then the integral 
\begin{equation*}
J_\bfy= \int_{\bfbet\in \R^R} I_\bfy(\bfbet)\d\bfbet
\end{equation*}
is absolutely convergent and we have
\begin{equation*}
|J_\bfy(\phitil(\bfy))-J_\bfy|\ll P_1^{\tet(R^2(d_1-1)-K)}|\bfy|^{R(R-1)d_2}.
\end{equation*}
Moreover, we have
\begin{equation*}
|J_\bfy|\ll |\bfy|^{R(R-1)d_2+\eps}.
\end{equation*}
\end{lemma}

\begin{proof}

Set $B=\max_i|\bet_i|$ for some real vector $\bfbet
\in \R^R$. Assume that we have $2B > C_3^{R(d_1-1)} |\bfy|^{(R-1)d_2}$. Then we choose the
parameters $0<\tet'\leq 1$ and $P$ in Lemma \ref{weyl2} in such a way that we have
\begin{equation*}
2 B = P^{R\tet' (d_1-1)}|\bfy|^{(R-1)d_2},
\end{equation*}
and
\begin{equation*}
P^{-K\tet'}= P^{-1+2R\tet' (d_1-1)}|\bfy|^{2Rd_2}.
\end{equation*}
In particular, this implies
\begin{equation*}
P^{-2+4R\tet' (d_1-1)}|\bfy|^{4Rd_2}<1,
\end{equation*}
and hence equation (\ref{MAdisj}) holds, since we have
assumed $d_1\geq 2$. Thus, the vector $P^{-d_1}\bfbet$ lies on the boundary
of the major arcs described in Lemma \ref{weyl2} and we therefore have the
estimate
\begin{equation*}
|S_\bfy(P^{-d_1}\bfbet)|< P^{n_1-K_1\tet' +\eps}.
\end{equation*}
On the other hand Lemma \ref{MAapprox} delivers 
\begin{equation*}
  P^{n_1}|I_\bfy (\bfbet)|\ll |S_\bfy(P^{-d_1}\bfbet)|  + O(P^{n_1-1+2R\tet' (d_1-1)}|\bfy|^{2Rd_2}).
\end{equation*}
Thus, we obtain the bound
\begin{equation*}
|I_\bfy(\bfbet)|\ll
B^{-K_1R^{-1}(d_1-1)^{-1}+\eps}|\bfy|^{K_1(R-1)d_2R^{-1}(d_1-1)^{-1}}.
\end{equation*}
Assume that $P_1^{\tet} > C_3$ with $P_1$ as in the assumptions of the lemma. This implies $2\phitil(\bfy) \geq C_3^{R(d_1-1)}|\bfy|^{(R-1)d_2}$. Thus we can estimate
\begin{align*}
|J_\bfy(\phitil(\bfy))-J_\bfy|&\ll \int_{B > \phitil(\bfy)}B^{R-1}
B^{-K_1R^{-1}(d_1-1)^{-1}+\eps}|\bfy|^{K(R-1)d_2R^{-1}(d_1-1)^{-1}}\d B\\ &\ll
\phitil(\bfy)^{R-K_1R^{-1}(d_1-1)^{-1}+\eps} |\bfy|^{K_1(R-1)d_2R^{-1}(d_1-1)^{-1}}\\
&\ll P_1^{\tet(R^2(d_1-1)-K_1)}|\bfy|^{R(R-1)d_2}, 
\end{align*}
which proves the first part of the lemma for $P_1$, which are greater than a
fixed constant depending on $\tet$. For the second part and small $P_1$ we note that the same computation delivers
\begin{equation*}
|J_\bfy(C_3^{R(d_1-1)}|\bfy|^{(R-1)d_2})-J_\bfy|\ll |\bfy|^{R(R-1)d_2+\eps},
\end{equation*}
and thus we obtain
\begin{equation*}
|J_\bfy|\ll |\bfy|^{R(R-1)d_2+\eps},
\end{equation*}
using the trivial estimate for $J_\bfy(C_3^{R(d_1-1)}|\bfy|^{(R-1)d_2})$.
\end{proof}

Next we prove similar results for the singular series $\grS_\bfy$ for $\bfy\in
\calA_1(\Z)$.

\begin{lemma}\label{lem4.4}
Let $\bfy\in \calA_1(\Z)$, and assume that we have $K_1> R(R+1)(d_1-1)$. Then the singular series
\begin{equation*}
\grS_\bfy(\phi(\bfy)) = \sum_{q\leq \phi(\bfy)}q^{-n_1} \sum_\bfa S_{\bfa,q}(\bfy)
\end{equation*}
is absolutely convergent and one has
\begin{equation*}
|\grS_\bfy(\phi(\bfy)) -\grS_\bfy| \ll P_1^{\tet(R(R+1)(d_1-1)-K+\eps)} |\bfy|^{d_2R(R+1)},
\end{equation*}
for some $\eps
>0$. Furthermore, one has the bound
\begin{equation*}
|\grS_\bfy|\ll |\bfy|^{d_2R(R+1)+\eps}.
\end{equation*}
\end{lemma}

\begin{proof}
Note that we have $S_{\bfa,q}(\bfy)=S_\bfy(\bfalp)$ for $P_1=q$ and
$\calB_1=[0,1)^{n_1}$ and $\bfalp =\bfa/q$. Assume that we are given some $q$
and $0< \tet' \leq 1$ with $q^{\tet'} >C_3$. Then, by Lemma \ref{weyl2} one has
either the upper bound
\begin{equation*}
|S_{\bfa,q}(\bfy)| < q^{n_1-K_1\tet'+\eps},
\end{equation*}
or there exist integers $q',a_1',\ldots, a_R'$ with $1\leq q'\leq q^{R\tet'(d_1-1)}|\bfy|^{Rd_2}$ and 
\begin{equation*}
2|q'a_i-a_i'q|\leq q^{1-d_1+R\tet' (d_1-1)}|\bfy|^{(R-1)d_2}
\end{equation*}
for all $1\leq i\leq R$. This is certainly impossible if $d_1\geq 2$ and $q^{R\tet' (d_1-1)}|\bfy|^{Rd_2} <q$.\par
Thus, for $q> C_3^{R(d_1-1)}|\bfy|^{Rd_2}$ we can choose $0< \tet'\leq 1$ by
$q^{R(\tet'+\eps)(d_1-1)}|\bfy|^{Rd_2}=q$, and obtain
\begin{equation*}
|S_{\bfa,q}(\bfy)|<
q^{n_1-K_1R^{-1}(d_1-1)^{-1}+\eps}|\bfy|^{K_1Rd_2R^{-1}(d_1-1)^{-1}}.
\end{equation*}
Next we note that for $P_1^{\tet}> C_3$ we have
$\phi(\bfy) > C_3^{R(d_1-1)}|\bfy|^{Rd_2}$, and hence we obtain the estimate
\begin{align*}
|\grS_\bfy(\phi(\bfy))-\grS_\bfy|&\ll \sum_{q> \phi(\bfy)} q^{-n_1}\sum_\bfa
|S_{\bfa,q}(\bfy)| \\ &\ll \sum_{q> \phi(\bfy)}
q^{R-K_1R^{-1}(d_1-1)^{-1}+\eps}|\bfy|^{K_1Rd_2R^{-1}(d_1-1)^{-1}} \\
&\ll
|\bfy|^{K_1Rd_2R^{-1}(d_1-1)^{-1}}P_1^{R\tet(d_1-1)(R+1-K_1R^{-1}(d_1-1)^{-1}+\eps)} \\ &\quad \times |\bfy|^{Rd_2(R+1-K_1R^{-1}(d_1-1)^{-1}+\eps)}\\
&\ll P_1^{\tet(R(R+1)(d_1-1)-K_1+\eps)} |\bfy|^{d_2R(R+1)}.
\end{align*}
For the second part of the lemma we use the same calculation, and obtain
\begin{align*}
|\grS_\bfy(C_3^{R(d_1-1)}|\bfy|^{Rd_2})-\grS_\bfy| &\ll
|\bfy|^{Rd_2(R+1-K_1R^{-1}(d_1-1)^{-1}+\eps)}\times
|\bfy|^{K_1Rd_2R^{-1}(d_1-1)^{-1}} \\ &\ll |\bfy|^{d_2R(R+1)+\eps}.
\end{align*}
We combine this with the trivial estimate
$|\grS_\bfy(C_3^{R(d_1-1)}|\bfy|^{Rd_2})|\ll|\bfy|^{d_2R(R+1)+\eps}$ to establish
the desired result.
\end{proof}

We put the results of this section together to prove an asymptotic formula for $N_\bfy(P_1)$.

\begin{lemma}\label{lem4.5}
Let $\bfy\in \calA_1 (\Z)$. Assume that we are given some $0<\tet \leq 1$ and $P_1\geq 1$ with $P_1^\tet >C_3$ and such that equation (\ref{MAdisj}) holds. Moreover, assume that we have
\begin{equation*}
K_1> (d_1-1)R(R+1).
\end{equation*}
Let $\Del (\tet, K_1)$ and $\eta (\tet)$ be defined as in Lemma \ref{lem3.2} and Lemma \ref{lem4.2}. Then we have the asymptotic formula
\begin{equation*}
N_\bfy(P_1)= \grS_\bfy J_\bfy P_1^{n_1-Rd_1} +O(E_2(\bfy))+O(E_3(\bfy)),
\end{equation*}
with
\begin{equation*}
E_2(\bfy)= P_1^{n_1-Rd_1-\eta (\tet)}|\bfy|^{2R(R+1)d_2},
\end{equation*}
and
\begin{equation*}
E_3(\bfy)= P_1^{n_1-Rd_1-\Del (\tet,K_1)+\eps}|\bfy|^{2R^2d_2}.
\end{equation*}
\end{lemma}

\begin{proof}
By Lemma \ref{lem4.2} we have 
\begin{equation*}
N_\bfy(P_1)= \grS_\bfy(\phi(\bfy)) J_\bfy(\phitil(\bfy)) P_1^{n_1-Rd_1} +O(E_1)+O(E_2),
\end{equation*}
with an error term
\begin{equation*}
E_1= P_1^{n_1-Rd_1-\Del (\tet,K_1)+\eps}|\bfy|^{R^2d_2}.
\end{equation*}
Hence we have $E_1\ll E_3$. By Lemma \ref{lem4.3} and \ref{lem4.4} we estimate
\begin{align*}
|\grS_\bfy(\phi(\bfy))J_\bfy(\phitil(\bfy))-\grS_\bfy J_\bfy| & \leq |\grS_\bfy (\phi(\bfy))-\grS_\bfy| |J_\bfy(\phitil(\bfy))|+|\grS_\bfy||J_\bfy(\phitil(\bfy))-J_\bfy| \\ & \ll P_1^{\tet (R(R+1)(d_1-1)-K_1+\eps)}|\bfy|^{R(R+1)d_2}|\bfy|^{R(R-1)d_2} \\ & + P_1^{\tet (R^2(d_1-1)-K_1+\eps)}|\bfy|^{R(R+1)d_2}|\bfy|^{R(R-1)d_2} \\ &\ll P_1^{\tet (R(R+1)(d_1-1)-K_1+\eps)}|\bfy|^{2R^2d_2},
\end{align*}
which proves the lemma.
\end{proof}

If we fix some small positive $\tet$ with $R(d_1-1)\tet< 1/(3+2R)$, then we obtain the following corollary.

\begin{corollary}\label{cor4.6}
Let $\bfy\in \calA_1(\Z)$, and assume that $K_1>R(R+1)(d_1-1)$. Then there is a $\del >0$, such that 
\begin{equation*}
N_\bfy(P_1)= \grS_\bfy J_\bfy P_1^{n_1-Rd_1} +O(P_1^{n_1-Rd_1-\del}|\bfy|^{2R(R+1)d_2}),
\end{equation*}
holds uniformly for all $|\bfy|<P_1^{\frac{d_1-1}{(3R-1)d_2}}$.
\end{corollary}

\begin{remark}
The results of this section still hold, if we take any system of homogeneous polynomials $F_{i,\bfb}(\bfx)$, with coefficients given by some integer vector $\bfb$, and replace $|\bfy|^{d_2}$ by $|\bfb|$ in the above lemmata.
\end{remark}

\section{Proof of Theorem \ref{thm1} and Theorem \ref{thm2}}
First we deduce Theorem \ref{thm1} from the lemmata that we have collected in the preceding sections.

\begin{proof}[Proof of Theorem \ref{thm1}]
First we note that by definition we have
\begin{equation*}
N_1(P_1,P_2)=\sum_{\bfy\in P_2\calB_2\cap \calA_1(\Z)}N_\bfy(P_1).
\end{equation*}
Hence, for some $\tet$ satisfying the assumptions of Lemma \ref{lem4.5}, we obtain
\begin{equation*}
N_1(P_1,P_2)=P_1^{n_1-Rd_1}\sum_{\bfy\in P_2\calB_2\cap \calA_1(\Z)}\grS_\bfy J_\bfy +O(\calE_2)+O(\calE_3),
\end{equation*}
with 
\begin{equation*}
\calE_2=\sum_{\bfy\in P_2\calB_2}E_2(\bfy),\quad \calE_3=\sum_{\bfy\in P_2\calB_2}E_3(\bfy).
\end{equation*}
Recall the notation $P_2=P_1^u$. Then we have
\begin{equation*}
\calE_2\ll P_1^{n_1-Rd_1}P_2^{n_2-Rd_2}P_1^{Rd_2u-\eta(\tet)+2R(R+1)d_2u},
\end{equation*}
and
\begin{equation*}
\calE_3 \ll P_1^{n_1-Rd_1}P_2^{n_2-Rd_2}P_1^{Rd_2u-\Del(\tet,K_1)+2R^2d_2u+\eps}.
\end{equation*}
Now we choose $\tet$ by
\begin{equation*}
Rd_2u-\eta (\tet)+2R(R+1)d_2u =-\del,
\end{equation*}
which is equivalent to saying that
\begin{equation*}
1-\del= (2R+3)Rd_2u+(2R+3)R\tet (d_1-1).
\end{equation*}
Note that this choice of $\tet$ is possible by the assumptions of Theorem \ref{thm1}, and it implies that equation (\ref{MAdisj}) holds. Moreover, this choice of $\tet$ ensures that the error term $\calE_2$ is sufficiently small.\par
Now, equation (\ref{eqnthm1}) implies that we have
\begin{equation*}
\tet (K_1-R(R+1)(d_1-1))> 2\del + Rd_2u+2R^2d_2u,
\end{equation*}
which leads to
\begin{equation*}
\calE_3\ll P_1^{n_1-Rd_1-\del}P_2^{n_2-Rd_2}.
\end{equation*}
This proves Theorem \ref{thm1} for $P_1^{\frac{1-\del-(2R+3)Rd_2u}{(2R+3)R(d_1-1)}}>C_3$.
\end{proof}

Recall that we have defined the counting function $N'(P_1,P_2)$ to be the number of integer
solutions $\bfx \in P_1\calB_1$ and $\bfy\in  P_2\calB_2$ to the system of
equations
\begin{equation*}
F_i(\bfx;\bfy)=0,
\end{equation*}
for $1\leq i\leq R$. We note that we have
\begin{equation*}
N'(P_1,P_2)= N_1(P_1,P_2)+O\left(\sum_{\bfy \in P_2\calB_2 \cap \calA_1^c(\Z)} P_1^{n_1}\right).
\end{equation*}
By Lemma \ref{lem2.4} these counting functions differ by at most
\begin{equation}\label{eqn5.1b}
N'(P_1,P_2)= N_1(P_1,P_2)+O(P_2^{n_2-\lam} P_1^{n_1}).
\end{equation}
As in section 4, we now choose $\lam =\lam_1 = \lceil R(b_1d_1+d_2)+\del\rceil$. Next we consider the case $P_1=P_2^{b_1}$, and note that
then we have
\begin{equation}\label{eqn5.2}
N'(P_1,P_2)= N_1(P_1,P_2)+O(P_2^{n_2-Rd_2 -\del} P_1^{n_1-Rd_1}).
\end{equation}
Assume additionally that we have 
\begin{equation*}
n_1+n_2 - \max \{\dim V_1^*, \dim V_2^*\}>
2^{d_1+d_2-2}R(b_1d_1+d_2).
\end{equation*}
Then the conditions on $n_1+n_2$ in Theorem
\ref{thm1} for $u=u_1$ and $\lam_1$ as above are equivalent to 
\begin{align*}
n_1+n_2 -\dim V_1^* > 2^{d_1-1}(g_1(u_1,\del)+R(R+1)(d_1-1))+\lceil R(b_1d_1+d_2)+\del\rceil.
\end{align*}

Thus, by definition of $b_1$, Theorem \ref{thm1} applies to our situation with $u=u_1$ and
delivers the asymptotic
\begin{equation}\label{eqn5.3}
N_1(P_1,P_2)= P_1^{n_1-Rd_1}\sum_{\bfy \in P_2\calB_2\cap \calA_1(\Z)} \grS_\bfy J_\bfy + O(P_1^{n_1-Rd_1-\del}P_2^{n_2-Rd_2}).
\end{equation}
Next we note that under the above assumptions Theorem \ref{thmbihom} delivers the asymptotic
\begin{equation}\label{eqn5.4}
N'(P_1,P_2) = \sig P_1^{n_1-Rd_1}P_2^{n_2-Rd_2} +
O(P_1^{n_1-Rd_1-\deltil}P_2^{n_2-Rd_2}),
\end{equation}
for some $\deltil >0$. A comparison of
equations (\ref{eqn5.2}), (\ref{eqn5.3}) and (\ref{eqn5.4}) shows that we have
\begin{equation}\label{eqn5.5}
\sum_{\bfy \in P_2\calB_2\cap \calA_1(\Z)} \grS_\bfy J_\bfy = \sig P_2^{n_2-Rd_2} +
O(P_2^{n_2-Rd_2-\deltil}).
\end{equation}
Note that this relation is independent of $P_1$, and thus holds for all
choices of $P_2$, as soon as $n_1+n_2 - \max \{\dim V_1^*, \dim V_2^*\}>
2^{d_1+d_2-2}R(b_1d_1+d_2)$. It is now easy to deduce the following proposition.  

\begin{theorem}\label{thm5}
Take $d_1,d_2\geq 2$, and let $n_1,n_2 >R$. Assume that
\begin{equation*}
n_1+n_2-\max \{ \dim V_1^*, \dim V_2^*\} > 2^{d_1+d_2-2} R(b_1 d_1+d_2).
\end{equation*}
Furthermore, let $\lam_1=\lceil R(b_1d_1+d_2)+\del\rceil$, and define the set $\calA_1(\Z)$ by
\begin{equation*}
\calA_1(\Z) = \{\bfz\in \Z^{n_2}: \dim V_{1,\bfz}^*< \dim V_1^* - n_2+\lam_1\}.
\end{equation*}
Assume $1\leq P_2\leq P_1$. Then there is some $\eps >0$, which is independent of $P_1$ and $P_2$ and the ratio of their logarithms, such that
\begin{equation*}
N_1(P_1,P_2) = \sig P_1^{n_1-Rd_1}P_2^{n_2-Rd_2} + O(
P_1^{n_1-Rd_1}P_2^{n_2-Rd_2-\eps}),
\end{equation*}
where $\sig$ is given as in Theorem \ref{thmbihom}.
\end{theorem}

\begin{proof}
Recall that we write $P_2=P_1^u$. First we consider the case $u\leq u_1$. The assumption 
\begin{equation*}
n_1+n_2-\max \{ \dim V_1^*, \dim V_2^*\} > 2^{d_1+d_2-2} R(b_1 d_1+d_2)
\end{equation*}
implies that
\begin{align*}
n_1+n_2- \max \{ \dim V_1^*, \dim V_2^*\} > & 2^{d_1-1}g_1(u_1,\del) +\lceil R(b_1d_1+d_2)+\del\rceil \\ &+2^{d_1-1}R(R+1)(d_1-1).
\end{align*}
By monotonicity of $g_1(u,\del)$ in the range of $0\leq u <u_1$ we thus obtain
\begin{equation*}
(K_1-R(R+1)(d_1-1))>g_1 (u_1,\del)\geq g_1(u,\del).
\end{equation*}
Hence Theorem \ref{thm1} is applicable and delivers 
\begin{equation*}
N_1(P_1,P_2)= P_1^{n_1-Rd_1}\sum_{\bfy \in P_2\calB_2\cap \calA_1(\Z)} \grS_\bfy J_\bfy + O(P_1^{n_1-Rd_1-\del}P_2^{n_2-Rd_2}).
\end{equation*}
Together with equation (\ref{eqn5.5}) this proves the theorem for $u\leq u_1$.\par
Next consider the case $u_1\leq u\leq 1$, i.e. $1\leq b\leq b_1$ if we write $b=1/u$. Note that by assumption we have
\begin{align*}
n_1+n_2-\max \{ \dim V_1^*, \dim V_2^*\} &> 2^{d_1+d_2-2} R(b_1 d_1+d_2)\\
&\geq 2^{d_1+d_2-2}R(bd_1+d_2).
\end{align*}
Furthermore we have $b_1>d_2(2R^2+3R)$ and hence
\begin{equation*}
n_1+n_2-\max \{\dim V_1^*,\dim V_2^*\} > 2^{d_1+d_2-2}R(R+1)(d_1+d_2-1).
\end{equation*}
Thus, we see that Theorem \ref{thmbihom} applies and delivers the asymptotic formula 
\begin{equation*}
N'(P_1,P_2)=\sig P_1^{n_1-Rd_1}P_2^{n_2-Rd_2}+O(P_1^{n_1-Rd_1-\eps}P_2^{n_2-Rd_2}).
\end{equation*}
By equation (\ref{eqn5.1b}) we have
\begin{equation*}
N'(P_1,P_2)=N_1(P_1,P_2)+O(P_2^{n_2-Rb_1d_1-Rd_2-\del}P_1^{n_1}),
\end{equation*}
which shows that the error in replacing $N'$ by $N_1$ is of acceptable size for $b\leq b_1$. 
\end{proof}

We can now prove Theorem \ref{thm2}.

\begin{proof}[Proof of Theorem \ref{thm2}]
Recall that we assume 
\begin{equation}\label{eqnthm2}
n_1+n_2 -\max \{\dim V_1^*,\dim V_2^*\} > 2^{d_1+d_2-2} R \max\{ (b_1d_1+d_2),
(b_2d_2+d_1)\}.
\end{equation}
Thus, the symmetric version of Theorem \ref{thm5} with the roles of $\bfx$ and $\bfy$ reversed implies that 
\begin{equation*}
N_2(P_1,P_2) = \sig P_1^{n_1-Rd_1}P_2^{n_2-Rd_2} + O(
P_1^{n_1-Rd_1-\deltil}P_2^{n_2-Rd_2}),
\end{equation*}
for $P_1\leq P_2$ and some $\deltil >0$. To prove Theorem \ref{thm2} it thus suffices to show that the error in replacing $N_1$ resp. $N_2$ by $N_U$ is small enough. For this we apply Lemma \ref{lem2.4}, and obtain
\begin{equation*}
|N_1(P_1,P_2)-N_U(P_1,P_2)|\ll \sum_{\bfx\in \calA_2^c(\Z)\cap P_1\calB_1}P_2^{n_2}\ll P_1^{n_1-\lam_2}P_2^{n_2}.
\end{equation*}
Recall that $\lam_2 =\lceil R(b_2d_2+d_1)+\del \rceil$ and $b_2\geq 1$. Hence the error is bounded by $\ll P_1^{n_1-Rd_1-\del}P_2^{n_2-Rd_2}$, for $P_2\leq P_1$. By symmetry the same applies to the difference $N_2(P_1,P_2)-N_U(P_1,P_2)$, in the case of $P_2\geq P_1$.
\end{proof}

\section{Transition to another height function and M\"obius inversion}
The first goal of this section is to apply the machine developed by Blomer and Br\"udern \cite{BloBru13} to the counting function $N_U(P_1,P_2)$. To make this precise we need to introduce some notation. Write $|\bfx|=\max_i|x_i|$ for the maximums norm. 
Let $h: \N^2\rightarrow [0,\infty)$ be an arithmetical function. Fix some real parameter $C$ and positive real parameters $\del$, $\bet_1$ and $\bet_2$. We say that $h$ satisfies condition (I) with respect to $(C,\del,\bet_1,\bet_2)$ if
\begin{equation*}
\sum_{\substack{l\leq L\\ m\leq M}}h(l,m)= CL^{\bet_1}M^{\bet_2}+O(L^{\bet_1}M^{\bet_2}\min\{L,M\}^{-\del}),
\end{equation*}
for all $L,M\geq 1$.
Fix further constants $\nu$ and $D$, where $\nu$ is positive and $D$
non-negative. We introduce a second condition for our arithmetical function
$h$.\\
(II) There exist arithmetical functions $c_1,c_2: \N\rightarrow [0,\infty)$
such that 
\begin{equation*}
\sum_{l\leq L}h(l,m) = c_1(m)L^{\bet_1}+O(m^D
L^{\bet_1-\del}),
\end{equation*}
holds uniformly for all $L\geq 1$ and $m\leq L^\nu$, and 
\begin{equation*}
\sum_{m\leq M}h(l,m)=c_2(l)M^{\bet_2}+O(l^D
M^{\bet_2-\del}),
\end{equation*}
holds uniformly for all $M\geq 1$ and $l\leq M^\nu$.\par
We say that a function $h$ is a $(C,\del,\bet_1,\bet_2,\nu,D)$-function if
it satisfies condition (I) and (II) with respect to these parameters.\par
We define the function 
\begin{equation*}
\Upsilon_h (P) = \sum_{l^{\bet_1}m^{\bet_2}\leq P}h(l,m).
\end{equation*}
A slight modification of Theorem 2.1 in \cite{BloBru13} yields the
following result.

\begin{theorem}\label{thm6.1}
Assume that $h$ is a $(C,\del,\bet_1,\bet_2,\nu,D)$-function. Then there is a
positive number $\eta$ and a real number $B$, such that one has the asymptotic formula
\begin{equation*}
\Upsilon_h (P) = C P \log P+BP +O(P^{1-\eta}).
\end{equation*}
\end{theorem}

We note that Theorem \ref{thm6.1} is not covered by Theorem 2.1 in
\cite{BloBru13} since for our application we will in general need $\bet_1\neq
\bet_2$. However, the proof of Theorem 2.1 in \cite{BloBru13} can easily be generalized
to our setting and is indeed much simpler since we only work with arithmetical
functions $h$ depending on two variables rather than $k$-dimensional functions
$h$ as in \cite{BloBru13}. We first define the counting function
\begin{equation*}
H(L,M)=\sum_{l\leq L}\sum_{m\leq M}h(l,m).
\end{equation*}

\begin{lemma}\label{lem6a}
Let $h$ satisfy condition (I) and (II). Then we have
\begin{equation*}
\sum_{l\leq L}c_2(l)=C L^{\bet_1}(1+O(L^{-\del})),
\end{equation*}
and
\begin{equation*}
\sum_{m\leq M}c_1(m)=CM^{\bet_2}(1+O(M^{-\del})).
\end{equation*}
\end{lemma}

\begin{proof}
By Condition (I) we have
\begin{equation*}
H(L,M)= CL^{\bet_1}M^{\bet_2}+O(L^{\bet_1}M^{\bet_2}\min \{L,M\}^{-\del}).
\end{equation*}
For $M\geq 1$ and $L\leq M^\nu$ Condition (II) implies
\begin{align*}
H(L,M)&= \sum_{l\leq L}\left( \sum_{m\leq M}h(l,m)\right) \\ 
&= \sum_{l\leq L}\left( c_2(l)M^{\bet_2}+O(l^DM^{\bet_2-\del})\right) \\
&= M^{\bet_2}\sum_{l\leq L}c_2(l) +O(L^{D+1}M^{\bet_2-\del}).
\end{align*}
Now choose $M=L^J$ for $J$ sufficiently large, such that $L\leq M^\nu$ and $L^{D+1}M^{-\del}=O(L^{\bet_1-\del})$. A comparison of both expressions for $H(L,M)$ yields
\begin{equation*}
\sum_{l\leq L}c_2(l)= CL^{\bet_1}+O(L^{\bet_1-\del}),
\end{equation*}
which proves the lemma.
\end{proof}

\begin{lemma}\label{lem6b}
Let $h$ satisfy Condition (I) and (II). Fix some $\mu$ with $0<\bet_1\mu <1/2$ satisfying
\begin{equation}\label{eqnlem6b1}
\mu (1+\nu \bet_1/\bet_2)\leq \nu/\bet_2,
\end{equation}
and
\begin{equation}\label{eqnlem6b2}
\mu (D-\bet_1+1+\del \bet_1/\bet_2)< \del / (2\bet_2).
\end{equation}
Define the sum
\begin{equation*}
T_1= \sum_{l\leq P^\mu}\sum_{P^{1/2}<m^{\bet_2}\leq P l^{-\bet_1}}h(l,m).
\end{equation*}
Then there is a real number $B'\in \R$ and some $\vartet >0$, such that we have
\begin{equation*}
T_1=\bet_1 C\mu P\log P+B' P+O(P^{1-\vartet}).
\end{equation*}
\end{lemma}

\begin{proof}
First note that we have
\begin{equation*}
T_1= \sum_{l\leq P^\mu}\sum_{l^{\bet_1}m^{\bet_2}\leq P}h(l,m)- H(P^{\mu},P^{1/(2\bet_2)}).
\end{equation*}
By our assumption (\ref{eqnlem6b1}) on $\mu$, we have
\begin{equation*}
l\leq \left( P^{1/\bet_2}l^{-\bet_1/\bet_2}\right)^{\nu},
\end{equation*}
for all $l\leq P^\mu$. Hence, by Condition (II), we obtain
\begin{equation*}
T_1=\sum_{l\leq P^{\mu}}\left(c_2(l)\left(\frac{P^{1/\bet_2}}{l^{\bet_1/\bet_2}}\right)^{\bet_2}+O\left( l^D \left(\frac{P^{1/\bet_2}}{l^{\bet_1/\bet_2}}\right)^{\bet_2-\del}\right)\right) - H(P^{\mu}, P^{1/{(2\bet_2)}}).
\end{equation*}
We have
\begin{equation*}
\sum_{l\leq P^\mu}l^{D-\bet_1+\del\bet_1/\bet_2}=O\left( P^{\mu (D-\bet_1+1+\del\bet_1/\bet_2)}+1\right),
\end{equation*}
which is bounded by $P^{\del/{(2\bet_2)}}$ by assumption (\ref{eqnlem6b1}) on $\mu$. Hence, we can express the sum under consideration as
\begin{equation*}
T_1= \left(\sum_{l\leq P^\mu}\frac{c_2(l)}{l^{\bet_1}}\right)P - H(P^\mu,P^{1/(2\bet_2)})+O(P^{1-\vartet}),
\end{equation*}
for some $\vartet >0$.\par
Next we evaluate $\sum_l c_2(l)/l^{\bet_1}$ via summing by parts. By Lemma \ref{lem6a} we can write
\begin{equation}\label{eins}
\sum_{l\leq L}c_2(l)= CL^{\bet_1}+E(L),
\end{equation}
with an error term of size at most $|E(L)|\ll L^{\bet_1-\del}$. Summing by parts leads us to
\begin{align*}
\sum_{l\leq P^\mu}\frac{c_2(l)}{l^{\bet_1}}= P^{-\mu \bet_1}\sum_{l\leq P^{\mu}}c_2(l) +\bet_1 \int_1^{P^\mu}t^{-\bet_1-1}\left(\sum_{l\leq t}c_2(l)\right) \d t.
\end{align*}
After inserting the asymptotic (\ref{eins}) we get
\begin{align*}
\sum_{l\leq P^\mu}\frac{c_2(l)}{l^{\bet_1}} &= P^{-\mu \bet_1}(CP^{\mu \bet_1}+O(P^{\mu \bet_1-\del\mu})) + \bet_1 \int_1^{P^{\mu}}t^{-\bet_1-1}(Ct^{\bet_1}+E(t))\d t \\ &= C+O(P^{-\vartet})+\bet_1C \log P^\mu + \bet_1 \int_1^{\infty} \frac{E(t)}{t^{\bet_1+1}}\d t +O\left( \int_{P^{\mu}}^\infty t^{-1-\del}\d t \right).
\end{align*}
Note that the integrals in the last line are both absolutely convergent by the bound on $E(L)$. Hence, we obtain
\begin{equation*}
\sum_{l\leq P^\mu}\frac{c_2(l)}{l^{\bet_1}}= \bet_1 C\mu \log P + B' +O(P^{-\vartet}),
\end{equation*}
for some real $B'$ and $\vartet >0$.\par
Note that by Condition (I) on the function $h$, we have
\begin{equation*}
H(P^{\mu},P^{1/(2\bet_2)})=O (P^{\bet_1\mu +1/2})=O(P^{1-\vartet}),
\end{equation*}
for some positive real $\vartet$. Putting these estimates into the expression for $T_1$, we finally obtain
\begin{equation*}
T_1=\bet_1 C \mu P\log P+ B' P+ O(P^{1-\vartet}),
\end{equation*}
which proves the lemma.
\end{proof}

We state the final lemma that we need for the proof of Theorem \ref{thm6.1}.

\begin{lemma}\label{lem6c}
Let $h$ be a function satisfying Condition (I), and assume that $0<\mu < \min \{ 1/(2\bet_1),1/(2\bet_2)\}$. Define the sum
\begin{equation*}
T_2= \sum_{P^\mu <l \leq P^{1/(2\bet_1)}}\sum_{P^{1/2}<m^{\bet_2}\leq P l^{-\bet_1}} h(l,m).
\end{equation*}
Then one has
\begin{equation*}
T_2= C(1/2-\bet_1\mu)P(\log P)+CP +O(P^{1/2+\bet_1\mu})+O(P^{1-(1/2)\mu\del}\log P).
\end{equation*}
\end{lemma}

\begin{proof}
Choose some large $J$, and define $\tet >0$ via 
\begin{equation*}
(1+\tet)^J = P^{1/(2\bet_1)-\mu}.
\end{equation*}
Consider numbers $P^\mu \leq L <L'\leq P^{1/(2\bet_1)}$ with $L'=
L(1+\tet)$. Define the slice
\begin{equation*}
V(L)= \sum_{L<l\leq L'} \sum_{P^{1/2}<m^{\bet_2}\leq
  P l^{-\bet_1}}h(l,m),
\end{equation*}
and the sums
\begin{equation*}
V_-(L) = \sum_{L<l\leq L'}\sum_{P^{1/2}<m^{\bet_2}\leq P (L')^{-\bet_1}}h(l,m),
\end{equation*}
and
\begin{equation*}
V_+(L) = \sum_{L<l\leq L'}\sum_{P^{1/2}<m^{\bet_2}\leq P L^{-\bet_1}}h(l,m).
\end{equation*}
By non-negativity of the function $h$ we obtain
\begin{equation}\label{drei}
V_-(L)\leq V(L)\leq V_+(L).
\end{equation}
Next we evaluate the sum $V_+(L)$. Note that by inclusion-exclusion we have
\begin{equation*}
V_+(L)= H(L',P^{1/\bet_2}L^{-\bet_1/\bet_2})-H(L', P^{1/(2\bet_2)}) - H(L,
P^{1/\bet_2}L^{-\bet_1/\bet_2})+H(L,P^{1/(2\bet_2)}).
\end{equation*}
Next consider the difference
\begin{align*}
H(L',P^{1/\bet_2}& L^{-\bet_1/\bet_2})-H(L,P^{1/\bet_2}L^{-\bet_1/\bet_2}) \\ &=
C((L')^{\bet_1}-L^{\bet_1})PL^{-\bet_1}+O((L')^{\bet_1}PL^{-\bet_1}\min \{
L',P^{1/\bet_2}L^{-\bet_1/\bet_2}\}^{-\del}).
\end{align*}
Since we have assumed $\mu <1/(2\bet_2)$, this expression equals
\begin{equation*}
C((1+\tet)^{\bet_1}-1)P+O((1+\tet)^{\bet_1}P^{1-\mu\del}).
\end{equation*}
Using $(1+\tet)^{\bet_1}= 1+\bet_1\tet +O(\tet^2)$, we get 
\begin{align*}
H(L',P^{1/\bet_2}L^{-\bet_1/\bet_2})-H(L,P^{1/\bet_2}L^{-\bet_1/\bet_2}) =
C\bet_1\tet P +O(P^{1-\mu\del})+O(\tet^2 P).
\end{align*}
Similarly, we obtain
\begin{align*}
H(L',P^{1/(2\bet_2)})-H(L,P^{1/(2\bet_2)})= C\bet_1\tet L^{\bet_1}P^{1/2}
+O(P^{1-\mu\del})+O(\tet^2 P).
\end{align*}
This gives the asymptotic
\begin{align*}
V_+(L)= C\bet_1\tet P+C \bet_1\tet L^{\bet_1}P^{1/2}+O(\tet^2 P)+
O(P^{1-\mu\del}).
\end{align*}
We assume from now on, that $\tet$ is sufficiently small and we will see in our choice of $J$ later that this is indeed the case. Using $(1+\tet)^{-\bet_1}= 1+O(\tet)$ for small $\tet$, a similar computation shows that we
have exactly the same asymptotic for $V_-(L)$, and hence for $V(L)$.\par
We now use a 'dyadic' decomposition in choosing
\begin{equation*}
L_j= P^\mu (1+\tet)^j,\quad 0\leq j<J.
\end{equation*}
The sum $T_2$, which we aim to evaluate, becomes
\begin{align*}
T_2&= \sum_{0\leq j<J}V(L_j) \\ 
&= C \bet_1 (J\tet)P+ C\bet_1 \tet P^{1/2}\sum_{0\leq j<J}L_j^{\bet_1} +
O(J\tet^2 P)+ O(JP^{1-\mu\del}).
\end{align*}
We compute
\begin{align*}
\tet \sum_{0\leq j<J}L_j^{\bet_1}&= \tet P^{\bet_1 \mu}
\frac{(1+\tet)^{J\bet_1}-1}{(1+\tet)^{\bet_1}-1} \\
&= P^{\bet_1\mu}\frac{P^{1/2-\bet_1\mu}-1}{\bet_1+O(\tet)} \\ 
&= \frac{1}{\bet_1}P^{1/2}+O(P^{\bet_1\mu})+O(P^{1/2}\tet).
\end{align*}
Therefore, we obtain
\begin{align*}
T_2= C \bet_1 (J\tet)P+ C P +O(P^{1/2+\bet_1\mu})+O(\tet P)+O(J\tet^2
P)+O(JP^{1-\mu\del}).
\end{align*}
Next we choose $J$ as the largest integer smaller than $P^{(1/2)\mu\del} \log P$. Note that by definition of $\tet$ we have
\begin{equation*}
J \log (1+\tet)= \left(\frac{1}{2\bet_1}-\mu\right) \log P,
\end{equation*}
and hence
\begin{equation*}
\tet = J^{-1}\left( \frac{1}{2\bet_1} - \mu \right) \log P +O(J^{-2}(\log P)^2).
\end{equation*}
This gives the asymptotic
\begin{equation*}
J\tet = \left(\frac{1}{2\bet_1}-\mu\right) \log P + O(P^{-\mu\del /2}(\log
P)),
\end{equation*}
and the bound $\tet =O(P^{-(1/2)\mu\del})$. Plugging this into the last expression
for $T_2$, we obtain
\begin{equation*}
T_2= C(1/2-\bet_1\mu) P(\log P) + CP+
O(P^{1/2+\bet_1\mu})+O(P^{1-(1/2)\mu\del}\log P).
\end{equation*}
\end{proof}

We can now give a proof of Theorem \ref{thm6.1}.

\begin{proof}[Proof of Theorem \ref{thm6.1}]
We start in writing
\begin{align*}
\Upsilon_h(P)&= \sum_{l^{\bet_1}m^{\bet_2}\leq P}h(l,m) \\
&= \sum_{\substack{l^{\bet_1}m^{\bet_2}\leq P\\
    m^{\bet_2}>P^{1/2}}} h(l,m)+  \sum_{\substack{l^{\bet_1}m^{\bet_2}\leq P\\
    l^{\bet_1}>P^{1/2}}} h(l,m)+H(P^{1/(2\bet_1)},P^{1/(2\bet_2)}).
\end{align*}
Note that 
\begin{equation*}
\sum_{\substack{l^{\bet_1}m^{\bet_2}\leq P\\
    m^{\bet_2}>P^{1/2}}} h(l,m)= T_1+T_2,
\end{equation*}
with $T_1$ and $T_2$ given in Lemma \ref{lem6b} and \ref{lem6c}. For $\mu$ sufficiently small these two lemmata together imply
\begin{equation*}
\sum_{\substack{l^{\bet_1}m^{\bet_2}\leq P\\
    m^{\bet_2}>P^{1/2}}} h(l,m)= (1/2) CP\log P +B'' P+O(P^{1-\eta}),
\end{equation*}
for some $B''\in \R$ and some positive real $\eta$. By symmetry, the same
asymptotic holds for the sum of $h(l,m)$ over all possible values $l^{\bet_1}m^{\bet_2}\leq P$ with $l^{\bet_1}>P^{1/2}$. Together with Condition (I) applied to
$H(P^{1/(2\bet_1)},P^{1/(2\bet_2)})$, this leads us to
\begin{equation*}
\Upsilon_h(P)= C P \log P +B P +O(P^{1-\eta}),
\end{equation*}
for some real number $B$, as desired.
\end{proof}

Our next goal is to apply Theorem \ref{thm6.1} to the following arithmetical
function. For some positive integers $l$ and $m$ let $h(l,m)$ be the
number of integer vectors $\bfx\in \Z^{n_1}$, $\bfy\in \Z^{n_2}$ with
$(\bfx;\bfy)\in U$ and $|\bfx|=l$ and $|\bfy|=m$ such that
$F_i(\bfx;\bfy)=0$ for all $1\leq i\leq R$.\par
Assume that equation (\ref{eqnthm2}) holds, i.e.
\begin{equation*}
n_1+n_2 -\max \{\dim V_1^*,\dim V_2^*\} > 2^{d_1+d_2-2} R \max\{ (b_1d_1+d_2),
(b_2d_2+d_1)\}.
\end{equation*}
Then Condition (I) for this function $h$ is directly provided by Theorem \ref{thm2} for $\calB_1=[-1,1]^{n_1}$ and $\calB_2=[-1,1]^{n_2}$ with respect to the parameters $C=\sig$, $\bet_1=n_1-Rd_1$, $\bet_2=n_2-Rd_2$
and $\del$ as given in Theorem \ref{thm2}.\par
It remains to verify Condition
(II). Recall that the open subset $U$ is by construction the
product of two open subsets $U_1\subset \A^{n_1}$ and $U_2\subset \A^{n_2}$,
i.e. $U=U_1\times U_2$. The sum $\sum_{l\leq L}h(l,m)$ counts all
integer vectors $(\bfx;\bfy)\in U$ such that $|\bfx|\leq L$ and $|\bfy|=m$
and $F_i(\bfx;\bfy)=0$, for all $1\leq i\leq R$. For fixed $\bfy$ let
$N_{\bfy,U}(L)$ be the number of integer solutions $|\bfx|\leq L$,
$\bfx\in U_1$ to the system of equations (\ref{eqn0}). Then we have
\begin{equation}\label{vier}
\sum_{l\leq L}h(l,m)= \sum_{|\bfy|= m, \bfy\in U_2}
N_{\bfy,U}(L).
\end{equation}
Fix some $\bfy\in U_2= \calA_1(\Z)$. Then equation (\ref{eqnthm2}) implies that 
\begin{equation*}
K_1>R(R+1)(d_1-1),
\end{equation*}
in the language of Corollary \ref{cor4.6} with $\lam= \lam_1$. Hence this corollary delivers an asymptotic formula 
\begin{equation*}
N_\bfy(L)= \grS_\bfy J_\bfy L^{n_1-Rd_1}+ O(L^{n_1-Rd_1-\del}|\bfy|^{2R(R+1)d_2}),
\end{equation*}
uniformly for $|\bfy|^{d_2}< L^{\frac{d_1-1}{3R-1}}$. We consider the difference of the counting functions $N_\bfy(L)$ and $N_{\bfy,U}(L)$. This is trivially bounded by the number of integer vectors $\bfx\in \calA_2^c(\Z)$ with $|\bfx|\leq L$. An application of Lemma \ref{lem2.4} to $\calA= \calA_2$ and $\lam=\lam_2$ delivers the bound
\begin{equation*}
\sharp \{\bfx\in \calA_2^c(\Z): |\bfx|\leq L\} \ll L^{n_1-\lam_2}.
\end{equation*}
Recall that we have defined $\lam_2= \lceil R(b_2d_2+d_1)+\del\rceil$. Hence we obtain
\begin{equation*}
|N_\bfy(L)- N_{\bfy,U}(L)|\ll L^{n_1-R d_1-\del},
\end{equation*}
which implies that we have the same asymptotic formula for $N_{\bfy,U}(L)$ as for $N_\bfy(L)$. We put these asymptotic formulas into equation (\ref{vier}) and set 
\begin{equation*}
c_1(m)= \sum_{|\bfy|= m, \bfy\in U_2}\grS_\bfy J_\bfy.
\end{equation*}
We obtain
\begin{align*}
\sum_{l\leq L}h(l,m)&= c_1(m)L^{n_1-Rd_1}+O \left( \sum_{|\bfy|=m} |\bfy|^{2R(R+1)d_2}L^{n_1-Rd_1-\del}\right) \\
&= c_1(m)L^{n_1-Rd_1}+O(m^{n_2-1+2R(R+1)d_2}L^{n_1-Rd_1-\del}),
\end{align*}
uniformly for all $m\leq L^{\frac{d_1-1}{(3R-1)d_2}}$. This verifies the first part of Condition (II) for the function $h$ with respect to the parameters
\begin{equation*}
D= n_2-1+2R(R+1)d_2,\quad \nu = \frac{d_1-1}{(3R-1)d_2}.
\end{equation*}
By symmetry, the same arguments prove the second part of Condition (II). Hence, the following corollary now follows
directly from Theorem \ref{thm6.1}

\begin{corollary}\label{cor8}
Assume that $d_1,d_2\geq 2$ and that equation (\ref{eqnthm2}) holds.
Let $h$ be given as above. Then we have the asymptotic formula
\begin{equation*}
\Upsilon_h(P)= \sig P\log P+ B P+O(P^{1-\eta}),
\end{equation*}
for some positive number $\eta>0$, and some $B\in \R$. 
\end{corollary}

We note that $\Upsilon_h(P)$ counts all integer vectors $(\bfx;\bfy)\in U$
with $|\bfx|^{\bet_1}|\bfy|^{\bet_2}\leq P$ and (\ref{eqn0}). Thus,
$\Upsilon_h(P)$ and $N_{U,H}(P)$ essentially only differ in whether or not they count
non-primitive vectors $\bfx$ and $\bfy$, i.e. solutions with $\gcd
(x_1,\ldots, x_{n_1})>1$ or $\gcd (y_1,\ldots, y_{n_2}) >1$. The last goal of this section is to apply a form of M\"obius inversion to the counting function $\Upsilon_h(P)$ to obtain an asymptotic formula for $N_{U,H}(P)$, and hence to prove Theorem \ref{thm3}.\par
We start with the observation that
\begin{align*}
N_{U,H} (P)= \frac{1}{4} \sum_{e_1^{\bet_1}e_2^{\bet_2}\leq P}\mu (e_1)\mu (e_2) \Upsilon_h\left( \frac{P}{e_1^{\bet_1}e_2^{\bet_2}}\right).
\end{align*}
In the following we assume that we have $\bet_i \geq 2$ for $i=1,2$. This is certainly true in the situation of Theorem \ref{thm3} since $\bet_i= n_i-Rd_i$ and $n_i$ is assumed to be sufficiently large by equation (\ref{eqnthm2}). Note that for $e_1^{\bet_1}e_2^{\bet_2}\leq P$ we can apply Corollary \ref{cor8} to the inner term, and obtain for $\eta <1/2$ the asymptotic formula
\begin{align*}
N_{U,H}(P) 
&= \frac{1}{4}\sig S_1P\log P - \frac{1}{4} \sig S_2P + \frac{1}{4} B S_1P +O\left( P^{1-\eta}\sum_{e_1,e_2}\left(\frac{1}{e_1^{\bet_1}e_2^{\bet_2}}\right)^{1-\eta}\right) \\
&= \frac{1}{4}\sig S_1P\log P - \frac{1}{4} \sig S_2P + \frac{1}{4} B S_1P +O( P^{1-\eta}),
\end{align*}
with 
\begin{equation*}
S_1= \sum_{e_1^{\bet_1}e_2^{\bet_2}\leq P} \frac{\mu (e_1)\mu(e_2)}{e_1^{\bet_1}e_2^{\bet_2}},
\end{equation*}
and
\begin{equation*}
S_2= \sum_{e_1^{\bet_1}e_2^{\bet_2}\leq P}\frac{\mu (e_1)\mu(e_2)}{e_1^{\bet_1}e_2^{\bet_2}}\log (e_1^{\bet_1}e_2^{\bet_2}).
\end{equation*}
We note that the appearing sums $S_1$ and $S_2$ are absolutely convergent. To be more precise, we have
\begin{align*}
S_1&= \frac{1}{\zet}(\bet_1)\frac{1}{\zet}(\bet_2)+ O\left(\sum_{e_1^{\bet_1}e_2^{\bet_2}\geq P}\frac{1}{e_1^{\bet_1}e_2^{\bet_2}}\right).
\end{align*} 
The error term is bounded by
\begin{align*}
\ll P^{-1/3} \sum_{e_1,e_2=1}^\infty \frac{1}{(e_1^{\bet_1}e_2^{\bet_2})^{2/3}}\ll P^{-1/3},
\end{align*}
since $\bet_1,\bet_2\geq 2$. Similarly, we have
\begin{align*}
S_2&= \frac{1}{\zet}(\bet_1)\sum_{e_2=1}^\infty \frac{\mu(e_2)}{e_2^{\bet_2}}\log (e_2^{\bet_2}) + \frac{1}{\zet}(\bet_2)\sum_{e_1=1}^\infty \frac{\mu (e_1)}{e_1^{\bet_1}}\log (e_1^{\bet_1})+ O(P^{-\eta})\\
&= \frac{1}{\zet(\bet_1)}\frac{\bet_2\zet'(\bet_2)}{\zet(\bet_2)^2}+\frac{1}{\zet(\bet_2)}\frac{\bet_1\zet'(\bet_1)}{\zet(\bet_1)^2}+O(P^{-\eta}),
\end{align*}
for some $\eta >0$, which finally proves Theorem \ref{thm3} for $d_1\geq 2 $ and $d_2\geq 2$.

\bibliographystyle{amsbracket}
\providecommand{\bysame}{\leavevmode\hbox to3em{\hrulefill}\thinspace}

\end{document}